\documentclass[11pt]{amsart}
\usepackage{mathrsfs}
\usepackage{amssymb}
\usepackage{amsmath,amscd}
\usepackage{amsfonts,amssymb}
\usepackage{latexsym}
\usepackage{amsbsy}

\usepackage{graphics,graphicx}
\usepackage[all]{xy}

\newtheorem{theorem}{Theorem}[subsection]
\newtheorem*{theorem*}{Theorem}

\newtheorem{lemma}{Lemma}[subsection]
\newtheorem{proposition}{Proposition}[subsection]
\newtheorem{corollary}{Corollary}[subsection]

\newtheorem*{conjecture*}{Conjecture}

\newtheorem*{remark*}{Remark}
\newtheorem{definition}{Definition}[subsection]

\newcommand{\bcen}{\begin{center}}      \newcommand{\ecen}{\end{center}}
\newcommand{\bay}{\begin{array}}      \newcommand{\eay}{\end{array}}
\newcommand{\beq}{\begin{eqnarray*}}      \newcommand{\eeq}{\end{eqnarray*}}
\def\lr#1{\langle #1\rangle}
\newcommand{\cal}{\mathcal}
\def\az{\alpha}    
\def\bz{\beta}     
\def\gz{\gamma}

\def\dz{\delta}  

\def\llz{\Lambda}      
\def\lz{\lambda}    
     \def\cp{{\cal P}} \def\cc{{\cal C}}
\def\ch{{\cal H}}   \def\ct{{\cal T}} \def\co{{\cal O}}

\def\ca{{\cal A}}
\def\cb{{\cal B}}
\def\cz{{\cal Z}}

\def\ct{{\cal T}}

\def\cn{{\cal N}}
\def\bbn{{\mathbb N}}  \def\bbz{{\mathbb Z}}  \def\bbq{{\mathbb Q}}
\def\bbf{{\mathbb F}}   \def\bbe{{\mathbb E}} \def\bbf{{\mathbb F}}
\def\fq{{\mathbb F}_q}
  \def\bbc{{\mathbb C}}
\def\bba{{\mathbb A}}

\def\te{{\tilde E}}
\def\bi{{\bf 1}}
   
\def\fk#1{{\mathfrak #1}}

\def\lra{\longrightarrow}  

\def\ra{\rightarrow}
\def\hom{\operatorname{Hom}}

\def\Im{\operatorname{Im}}
\def\ext{\operatorname{Ext}}
\def\aut{\operatorname{Aut}}

\def\dim{\operatorname{dim}}
\def\cdim{\operatorname{codim}}
\def\min{\operatorname{min}}
\def\udim{\operatorname{\underline {dim}}}
\def\ed{\operatorname{End}}
\def\Hom{\operatorname{Hom}}
\def\Ext{\operatorname{Ext}}
\def\ol{\overline}
\def\mod{\operatorname{mod}}  
\def\Im{\operatorname{Im}}

\def\uq2{U_q(\hat{sl}_2)}
\def\fq{\bbf_q}

\def\wm{{\bf m}}

\def\nd{{\noindent}}
\def\mk{{\medskip}}
\def\hs{{\hskip 1cm}}

\def\det{\operatorname{det}}

\title[Canonical bases]{ Rational Ringel-Hall algebras, Hall polynomials of affine type and Canonical bases}

\thanks{The research was
supported in part by  NSCF grant 10931006 and by STCSM(12ZR1413200)  \\
${\quad}$The research was also supported in part by Research
Institute for Mathematical Sciences, Kyoto University, Kyoto, California Institute of Technology, Pasadena, California and Chern Institute of Mathematics, Nankai University, Tianjin,China}
\author{Guanglian Zhang}
\address{Department of Mathematics\\
Shanghai Jiao Tong University\\
Shanghai 200240, China}
 \email{g.l.zhang@sjtu.edu.cn}

\date{\today}

\begin{document}

\begin{abstract}
In this paper, the rational Ringel-Hall algebras for tame quivers are
introduced and are identified with the positive part of the
quantum extended Kac-Moody algebras. By using the rational Ringel-Hall algebras, we show that the existence of Hall polynomials for tame quiver algebras.
The PBW bases are constructed and new classes of perverse sheaves are shown to have strong purity property. These
allows us to construct the canonical bases of the positive part of the
quantum extended Kac-Moody algebras.
\end{abstract}

\maketitle

\bigskip
\bigskip

\centerline{\bf 0. Introduction}

\bigskip
The early nineties of the last century witnessed the invention of
canonical bases, Lusztig claimed he found it first, and now is commonly accepted~\cite{L1,L2,L3}. Since then, canonical bases have been playing a
vital role in representation theory of quantum groups, Hecke
algebras, and quantized $q-$Schur algebras ~\cite{A,BN,LXZ,F1,F2,
Ro,VV}.

Before that, Ringel--Hall algebras (of abelian categories)
were introduced by Ringel~\cite{R2} in order to obtain a categorical
version of Gabriel's theorem~\cite{BGP}. Ringel~\cite{R1} showed
that the Ringel--Hall algebra of the category of finite-dimensional
representations over a finite field of a Dynkin quiver, after a
twist by the Euler form, is isomorphic to the positive part of the
quantum group associated to the underlying graph of the quiver. This
was generalised by Green~\cite{G} to all quivers, with the whole
Ringel--Hall algebra replaced by its subalgebra generated by the
1-dimensional representations (the composition algebra).

The above relationship between Ringel--Hall algebras and quantum
groups enabled Lusztig to give geometric construction of canonical
bases by using quiver varieties~\cite{L1,L2,L3}. For finite types he
also has an algebraic construction ~\cite{L4}, and this was
generalised to affine types by Nakajima-Back in~\cite{BN}, and by Lin, Xiao and the author
in~\cite{LXZ}.

In this paper we study canonical bases for quantum extended affine
Kac-Moody algebras. We follow Lusztig's geometric approach. For this
purpose, we introduce the \emph{ratinal Ringel--Hall algebras $\ch^{r, x_1,\cdots, x_t}(Q)$}.
Thanks to a theorem of Sevenhant and Van den Bergh~\cite{SV}, the
whole Ringel--Hall algebra $\ch(Q)$ of a tame quiver $Q$ can be
obtained from the composition algebra $\cc(Q)$ by adding some
imaginary simple roots. The rational Ringel--Hall algebra $\ch^r(Q)$
will be defined as a subalgebra of $\ch(Q)$, obtained from $\cc(Q)$
by adding certain imaginary simple roots. It will be shown that
$\ch^r(Q)$ is isomorphic to the positive part of the quantum
extended Kac-Moody algebra of the underlying graph of $Q$. To
construct the canonical bases, we start with constructing a PBW
basis and, as in ~\cite{L1}, proving that certain perverse sheaves
over the associated quiver varieties have the strong purity property. It
then follows that certain elements in the rational Ringel--Hall
algebra associated to these sheaves form the canonical bases. All the results for $\ch^r(Q)$ we have obtained are also true for the other ratinal Ringel--Hall algebras $\ch^{r, x_1,\cdots, x_t}(Q)$. In
\cite {KS}, Kang and Schiffmann obtained canonical bases of quantum
generalized Kac-Moody algebras via construct a class quiver with
edge loops. An important difference between our construction and
Kang and Schiffmann's one is that our canonical bases are
made of simple perverse sheaves rather than semisimple perverse
sheaves. That is, our canonical bases coincide with classical
canonical bases of Lusztig's version.

This work is another motivated by the application of canonical bases
of the Fock space over the quantum affine $\mathfrak{sl}_n$ in the
study of decomposition numbers: the Lascoux--Leclerc--Thibon
conjecture ~\cite{LLT} which was solved by Ariki~\cite{A}, and the
Lusztig conjecture for $q$-Schur algebras by Varagnolo and
Vasserot~\cite{VV}. These Fock spaces are of type $A$ and level $1$.
Those of type $A$ and of higher levels are studied in~\cite{A,KS}
and other papers. Our aim is to solve Lusztig's conjecture for
$q$-Schur algebras of other types. The first step is to find
appropriate Fock spaces. We propose that the quantum extended affine
Kac--Moody algebra modulo certain elements corresponding to unstable
orbits in Nakajima's sense, considered as a module over the quantum
affine algebra, is the correct Fock space. For type $A$, this is verified.

 The paper is organized as follows. In \S 1 we
 give a quick review of the definitions of Ringel-Hall algebras and Double Ringel-Hall algebras.
In \S 2 we define the rational Ringel-Hall algebras $\ch^{r, x_1, \cdots, x_t }(\llz)$
and study the relations between the rational Ringel-Hall algebras
and the quantized universal enveloping algebras. We prove
Proposition~2.2.1 and from this, we point out that the quotient of
the rational Ringel-Hall algebras $\ch^{r}(\llz)$ modulo unstable orbits is
isomorphic to $q-$ Fock space in the case of type $A$. In \S 3 we
show that  the rational Ringel-Hall algebra $\ch^r(\llz)$ is
isomorphic to the positive part  $\mathbb{U}^+$ of the quantum
extended Kac-Moody algebra. In \S 4 we construct the $PBW$ type
basis for the positive part  $\mathbb{U}^+$ of the quantum extended
Kac-Moody algebras. In \cite{H}, A.Hubery proved the existence
of Hall polynomials on the tame quivers for Segre classes. Since the work of Hubery, I am  the first  to give a complete proof for the existence of
Hall polynomials of the tame quivers, which is arranged in section 5.
In \S 6 we give the main theorem of this paper, that is, a description of the canonical basis of $\ch^r(\llz)\cong\mathbb{U}^+$. In $\S 7$ we prove that the closure of
semi-simple objects in $\ct_i$ have strong purity property.
In $\S 8$ we give the proof of Theorem 6.1.1. From Lemma 8.1.1, we point out that the basis of $\mathbf{U}_n^{+}$ defined by \cite{VV} is Lusztig's canonical bases of $\mathbf{U}_n^{+}.$
Moreover, using our new canonical basis, we prove that the coefficients of
Hall polynomial $\varphi^{X_3}
 _{X_1X_2}$ in section 5, are integers.

\mk\nd{\bf Acknowledgments.} I would like to express my sincere
gratitude to Hiraku Nakajima, Xinwen Zhu and Jie Xiao for a number of interesting discussions.


\section{The Ringel--Hall algebra}\label{s:hall-alg}

Throughout the paper, let $\bbf_q$ denote a finite field with $q$
elements, and $k=\overline{\bbf}_q$ be the algebraic closure of $\bbf_q$.

\subsection{}\label{ss:1.1} A \emph{quiver} $Q=(I,H,s,t)$ consists of  a
vertex set $I$,  an arrow set $H$, and two maps $s,t:H\rightarrow I$
such that an arrow $\rho\in H$ starts at $s(\rho)$ and terminates at
$t(\rho).$ The two maps $s$ and $t$ extends naturally to the set of paths.
A \emph{representation} $V$ of $Q$ over a field $F$ is a collection $\{V_i:i\in I\}$ of $F$-vector spaces
and a collection $\{V(\rho):V_{s(\rho)}\rightarrow V_{t(\rho)}:\rho\in H\}$ of $F$-linear maps.

Let $\llz=\bbf_q Q$ be the \emph{path algebra} of $Q$ over the field $\bbf_q$.
Precisely, $\llz$ is a $\bbf_q$-vector space with basis all paths of $Q$ (including the trivial paths attached to all vertices), and the product $pq$ of two paths $p$ and $q$ is the concantenation of $p$ and $q$ if $s(p)=t(q)$ and is $0$ otherwise.
The
category of all finite-dimensional left $\llz$-modules, namely finite left $\llz$-modules, is equivalent to the category of finite-dimensional representations of
$Q$ over $\bbf_q.$ We shall simply identify $\llz$-modules with
representations of $Q$, and call a $\llz$-module \emph{nilpotent} if the corresponding representation is nilpotent. The category of nilpotent $\llz$-modules will be denoted by $\mod\llz$.

The set of isomorphism classes of nilpotent simple $\llz$-modules
is naturally indexed by the set $I$ of vertices of $Q.$ Hence the
Grothendieck group $G(\llz)$ of $\mod\llz$ is the free Abelian group
$\bbz I$. For each nilpotent $\llz$-module $M$, the dimension vector
$\udim M=\sum_{i\in I}(\dim M_i)i$ is an element of $ G(\llz)$.

The Euler form $\lr{-,-}$ on $G(\llz)=\bbz I$ is defined by
\beq\hs \lr{\az,\bz}=\sum_{i\in I}a_ib_i-\sum_{\rho\in
H}a_{s(\rho)}b_{t( \rho)}\eeq for $\az=\sum_{i\in I}a_i i$ and
$\bz=\sum_{i\in I}b_i i$ in $\bbz I.$ For any nilpotent $\llz$-modules $M$ and $N$ one has
$$\lr{\udim M, \udim
N}=\dim_{\fq}\hom_{\llz}(M,N)-\dim_{\fq}\ext_{\llz}(M,N).$$ The
symmetric Euler form is defined as
$$(\az,\bz)=\lr{\az,\bz}+\lr{\bz,\az}\ \ \text{for}\ \
\az,\bz\in\bbz I.$$ This gives rise to a symmetric generalized
Cartan matrix $C=(a_{ij})_{i,j\in I}$ with $a_{ij}=(i,j).$ It is
easy to see that $C$ is independent of the field $\fq$ and the
orientation of $Q.$


\subsection{The Ringel--Hall algebra}\label{ss:hall-alg}
Let $Q$ be a quiver, and $\bbf_qQ$ be the path algebra of $Q$ over the finite field $\bbf_q$.

Let $v=v_q=\sqrt q\in \bbc$ and $\cp$
 be the set of isomorphism classes of finite-dimensional nilpotent $\llz$-modules.
The \emph{(twisted) Ringel--Hall algebra} $\ch^*(\llz)$ is defined as the $\bbq(v)$-vector space with basis $\{u_{[M]}:[M]\in\cp\}$ and with multiplication  given by
$$u_{[M]}\ast
u_{[N]}=v^{\lr{\udim M, \udim
N}}\sum_{[L]\in\cp}g^{L}_{MN}u_{[L]}.$$
where $g_{MN}^{L}$ is the number of $\llz$-submodules
$W$ of $L$ such that $W\simeq N$ and $L/W\simeq M$ in $\mod \llz$. Its subalgebra generated by $\{u_{[M]}:M \text{ is a simple nilpotent } \Lambda \text{ module}\}$ is called the \emph{composition subalgebra} of $\ch^*(\llz)$ or the \emph{composition Ringel--Hall algebra} of $\llz$, and denoted by $\cc^*(\llz)$.
The Ringel--Hall algebra $\ch^*(\llz)$ and the composition algebra $\cc^*(\llz)$ is graded by $\bbn I$, namely, by dimension vectors of modules, since $g_{MN}^L\neq 0$ if and only if there is a short exact sequence $0\rightarrow N\rightarrow L\rightarrow M\rightarrow 0$, which implies that $\udim L =\udim M+\udim N$.
Following [R3], for any nilpotent
$\llz$-module $M$, we denote $$\lr{M}=v^{-\dim
 M+\dim\ed_{\llz}(M)}u_{[M]}.$$  Note that $\{ \lr{M} \;|\; M \in
\cp\}$ is  a $\bbq(v)$-basis of $\ch^*(\llz)$.

The $\bbq(v)$-algebra $ \ch^*(\llz)$ depends on $q(=v^2)$. We will use
$ \ch_q^*(\llz)$ to indicate the dependence on $q$ when such a need
arises.

\subsection{A construction by Lusztig}\label{ss:construction-lusztig}
Let  $V=\bigoplus_{i\in I}V_i$ be a finite-dimensional
$I$-graded $k$-vector space with a given
$\bbf_q$-rational structure by the Frobenius map $F:k\rightarrow k, a\mapsto a^q$. Let $\bbe_V$ be
the subset of $\bigoplus_{\rho\in H}\hom(V_{s(\rho)},V_{t(\rho)})$
consisting of elements which define nilpotent representations of $Q$. Note that
$\bbe_V=\bigoplus_{\rho\in H}\hom(V_{s(\rho)},V_{t(\rho)})$ when $Q$
has no oriented cycles. The space of $\bbf_q$-rational points of
$\bbe_V$ is the fixed-point set $\bbe^F_V.$

Let $G_V=\prod_{i\in I}GL(V_i)$, and its subgroup of
$\bbf_q-$rational points be $G^F_V$. Then the group $G_V=\prod_{i\in
I}GL(V_i)$ acts naturally on  $\bbe_V$ by
$$(g,x)\mapsto g\bullet x=x'\ \ \text{where}\ \
x'_{\rho}=g_{t(\rho)}x_{\rho}g^{-1}_{s(\rho)}\ \ \text{for all}\ \
\rho\in H.$$ For $x\in\bbe_V$, we denote by $\co_x$ the orbit of $x$.
This action restricts to an action of the finite group $G^F_V$ on $\bbe^F_V.$

For $\gz\in\bbn I,$ we fix a $I$-graded $k$-vector space $V_{\gz}$
with $\udim V_{\gz}=\gz.$ We set $\bbe_{\gz}=\bbe_{V_{\gz}}$ and
$G_{\gz}=G_{V_{\gz}}.$ For $\az,\bz\in\bbn I$ and $\gz=\az +\bz,$ we
consider the diagram
$$\xymatrix{\bbe_{\az}\times\bbe_{\bz} & \bbe' \ar[l]_(0.35){p_1}\ar[r]^{p_2}&\bbe''\ar[r]^{p_3}& \bbe_{\gz}.}$$ Here $\bbe''$ is the set of all pairs $(x,W)$,
consisting of $x\in\bbe_{\gz}$ and an $x$-stable $I$-graded subspace
$W$ of $V_{\gz}$ with $\udim W=\bz$, and $\bbe'$ is the set of all
quadruples $(x,W,R',R'')$, consisting of $(x,W)\in\bbe''$ and two
invertible linear maps $R':k^{\bz}\ra W$ and $R'':k^{\az}\ra
k^{\gz}/W.$  The maps are defined in an obvious way:
$p_2(x,W,R',R'')=(x,W),$ $p_3(x,W)=x,$ and
$p_1(x,W,R',R'')=(x',x''),$ where
$x_{\rho}R'_{s(\rho)}=R'_{t(\rho)}x'_{\rho}$ and
$x_{\rho}R''_{s(\rho)}=R''_{t(\rho)}x''_{\rho}$ for all $\rho\in H.$
By Lang's Lemma, the varieties and morphisms in this diagram are
naturally defined over $\bbf_q.$ So we have
$$\xymatrix{\bbe^F_{\az}\times\bbe^F_{\bz} & \bbe'^F \ar[l]_(0.35){p_1}\ar[r]^{p_2}&\bbe''^F\ar[r]^{p_3}& \bbe^F_{\gz}.}$$
For $M\in\bbe_{\az}, N\in\bbe_{\bz}$ and $L\in\bbe_{\az+\bz},$ we
define
$${\bf Z}=p_2p_1^{-1}(\co_M\times \co_N)\subseteq E'',\qquad {\bf Z}_{L,M,N}={\bf Z}\cap p_3^{-1}( \co_L).$$

For any map $p: X\rightarrow Y$ of finite sets,
$p^*:\bbc(Y)\rightarrow \bbc(X)$ is defined by $p^*(f)(x)=f(p(x))$
and $p_!: \bbc(X)\rightarrow \bbc(Y)$ is defined by $
p_!(h)(y)=\sum_{x \in p^{-1}(y)}h(x),$ on the integration along the
fibers. Let $\bbc_{G^F}(\bbe^F_V)$ be the space of $G^F_V$-invariant
functions $\bbe^F_V\ra \bbc ( \text{ or } \overline{Q_l}).$ Given
$f\in\bbc_{G^F}(\bbe^F_\az)$ and $g\in\bbc_{G^F}(\bbe^F_\bz)$, there
is a unique $h\in \bbc_G(\bbe''^F)$ such that
$p_2^*(h)=p_1^*(f\times g).$ Then define $f\circ g$ by
$$f\circ g=(p_3)_!(h)\in\bbc_{G^F}(\bbe^F_{\gz}).$$

 Let $${\wm}(\az,\bz)= \sum_{i\in
I}a_ib_i+\sum_{\rho\in H}a_{s(\rho)}b_{t( \rho)}.$$ We again define
the multiplication in the $\bbc$-space ${\bf K}=\bigoplus_{\az\in\bbn
I}\bbc_{G^F}(\bbe^F_{\az})$ by
$$f\ast g=v_q^{-{\wm}(\az,\bz)}f\circ g$$ for all $f\in\bbc_{G^F}(\bbe^F_\az)$ and
$g\in\bbc_{G^F}(\bbe^F_\bz).$ Then $({\bf K},\ast)$ becomes an
associative $\bbc$-algebra.

For $M\in\bbe^F_{\az}$, let  $\co_M\subset\bbe_{\az}$ be the
$G_{\az}$-orbit of $M.$ We take
$\bi_{[M]}\in\bbc_{G^F}(\bbe^F_{\az})$ to be the characteristic
function of $\co^F_M,$ and set $f_{[M]}=v_q^{-\dim\co_M}\bi_{[M]}.$
We consider the subalgebra $({\bf L},\ast)$ of $({\bf K},\ast)$
generated by $f_{[M]}$ over $\bbq(v_q),$ for all $M\in\bbe^F_{\az}$
and all $\az\in\bbn I.$ In fact ${\bf L}$ has a $\bbq(v_q)$-basis
$\{f_{[M]}|M\in\bbe^F_{\az}, \az\in\bbn I\}.$ Since  $\bi_{[M]}\circ
\bi_{[N]}(W)=g^W_{M N}$ for any $W\in\bbe^F_{\gz},$ we have


\begin{proposition}\cite{LXZ}\label{p:1.3.1}
The linear map
$\varphi:({\bf L},\ast)\lra \ch^*(\llz)$ defined by
$$\varphi(f_{[M]})=\lr{M},\ \ \ \text{for all}\ [M]\in\cp$$
is an isomorphism of associative $\bbq(v_q)$-algebras.
\end{proposition}

\subsection{The double Ringel--Hall algebra
$\mathscr{D}(\llz)$}\label{ss:double-hall-alg} First, we define a
Hopf algebra $\bar\ch^+(\llz)$ which is a $\bbq(v)$-vector space with
the basis $\{K_{\mu}u_{\alpha}^+|\mu\in\bbz[I],\az\in\cp\},$ whose
Hopf algebra structure is given as

(a) Multiplication (\cite{R1})
\begin{eqnarray*}u_{\az}^+*u_{\bz}^+&=&v^{\lr{\az,\bz}}\sum_{\lz\in\cp}g^{\lz}_{\az\bz}u^+_{\lz},
\text{ for all } \az,\bz\in\cp,\\
K_{\mu}*u_{\az}^+&=&v^{(\mu,\az,)}u_{\az}^+*K_{\mu},\text{ for all
} \az\in\cp,
\mu\in\bbn[I],\\
K_\mu*K_\nu&=&K_\nu*K_\mu\hspace{7pt}=\hspace{7pt}K_{\mu+\nu}, \text{ for all }
\mu,\nu\in\bbn[I].
\end{eqnarray*}

(b)Comultiplication (\cite{G})
\begin{eqnarray*}
\bigtriangleup(u_{\lz}^+)&=&\sum_{\az,\bz\in\cp}v^{\lr{\az,\bz}}
\frac{a_{\az}a_{\bz}}{a_{\lz}}g^{\lz}_{\az\bz}u^+_{\az}K_{\bz}\otimes
u_{\bz}^+, \text{ for all } \lz\in\cp,\\
\bigtriangleup(K_{\mu})&=&K_\mu\otimes K_\mu,\text{ for all } \mu\in\bbn[I]
\end{eqnarray*}
with counit $\epsilon(u_{\lz}^+)=0,$ for all $0\neq \lz\in\cp$, and
$\epsilon(K_\mu)=1.$ Here $a_\lz$ denotes the cardinality of the
finite set $\aut_{\llz}(M)$ with $[M]=\lz.$

(c)Antipode(\cite{X})
\begin{eqnarray*}
   S(u^+_{\lz})&=&\dz_{\lz0}+\sum(-1)^m\sum_{\pi\in\cp,\lz_1,\cdots,\lz_m\in\cp\setminus \{0\}}\times
\\&&\mbox{}v^{2\sum_{i<j}\lr{\lz_i,\lz_j}}\frac{a_{\lz_1}\cdots a_{\lz_m}}{a_{\lz}}g^\lz_{\lz_1\cdots\lz_m}
g^\pi_{\lz_1\cdots\lz_m}K_{-\lz}u^+_\pi, \text{ for all } \lz\in\cp,\\
S(K_\mu)&=&K_{-\mu} \text{  for all } \mu\in\bbz[I].
\end{eqnarray*}
 The subalgebra $\ch^+$ of $\bar\ch^+(\llz)$  generated by $\{u_\lz|\lz\in\cp\}$ is isomorphic to $\ch^*(\llz).$ Moreover, we have an isomorphism of vector spaces $\bar\ch^+(\llz)\cong \ct\otimes \ch^+$ , where $\mathcal {T}$ denotes the torus subalgebra generated by
$\{K_\mu:\mu\in\bbz[I]\}.$

  Dually, we can define a Hopf algebra $\bar\ch^-(\llz)$. Following
  Ringel, we have a bilinear form
  $\varphi:\bar\ch^+(\llz)\times\bar\ch^-(\llz){\lra}\bbq(v)$ defined by
  $$\varphi(K_\mu u_\az^+,K_\nu u^-_{\bz})=v^{-(\mu,\nu)-(\az,\nu)+(\mu,\bz)}\frac{|V_\az|}{a_\az}\dz_{\az\bz}$$
for all $\mu,\nu\in\bbz[I]$ and all $\az,\bz\in \cp.$ Thanks to
\cite{X}, we can form the reduced Drinfeld double
$\mathscr{D}(\llz)$ of the Ringel--Hall algebra of $\llz$, which admits a
triangular decomposition
$$\mathscr{D}(\llz)=\ch^-\otimes\mathcal {T}\otimes\ch^+.$$
It is a Hopf algebra, and the restriction of this structure on
$\bar\ch^-(\llz)=\ch^-\otimes\ct$ and $\bar\ch^+(\llz)=\ct\otimes\ch^+$ are given as above.

The subalgebra of $\mathscr{D}(\llz)$ generated by
$\{u_i^{\pm},K_{\mu}|i\in I, \mu\in\bbz[I]\}$ is also called the
\emph{composition algebra} of $\llz$ and denoted by $\mathcal {C}(\llz).$ It is
 a Hopf subalgebra of $\mathscr{D}(\llz)$ and admits a triangular decomposition
$$\mathcal {C}(\llz)=\mathcal {C}^{-}(\llz)\otimes\mathcal {T}\otimes\mathcal {C}^+(\llz),$$
where $\mathcal {C}^+(\llz)$ is the subalgebra generated by
$\{u_i^+:i\in I\}$ and $\mathcal {C}^{-}(\llz)$ is defined dually.
Moreover, the restriction $\varphi:\mathcal
{C}^+(\llz)\times\mathcal {C}^{-}(\llz){\lra}\bbq(v)$ is
non-degenerate (see\cite{HX}).

 In addition, $\mathscr{D}(\llz)$ admits an involution $\omega$
 defined by
\begin{eqnarray*}
\omega(u_{\lz}^+)=u_{\lz}^{-},~~ \omega(u_{\lz}^{-})=u_{\lz}^{+},
\text{ ~for all ~} \lz\in\cp;
\\
\omega(K_{\mu})=K_{-\mu},\text{ for all } \mu\in\bbn[I].\qquad\qquad
\end{eqnarray*}
We have $\varphi(x,y)=(\omega(x),\omega(y)).$ Obviously, $\omega$
induces an involution of $\mathcal {C}(\llz).$

\subsection{Tame quivers}\label{ss:tame-quiver}
Let $Q=(I,H,s,t)$ be a connected tame quiver, that is, a quiver whose underlying graph is an extended Dynkin diagram
of type $\tilde{A}_n$, $\tilde{D}_n$, $\tilde{E}_6$, $\tilde{E}_7$ or $\tilde{E}_8$, and which has no oriented cycles,
i.e. in $Q$ there are no paths $p$ with $s(p)=t(p)$.
We say that $Q$ is of type $\tilde{A}_{p,q}$ if the underlying graph of $Q$ is of type $\tilde{A}_{p+q-1}$ and there are $p$ clockwise oriented arrows and $q$ anti-clockwise oriented arrows.

Let $\Lambda=\fq Q$ be the path algebra of $Q$ over $\fq$.
Any (nilpotent) finite-dimensional $\Lambda$-module is a direct sum of modules of three types: preprojective, regular, and preinjective. The set of isomorphism classes of preprojective modules (respectively, preinjective modules) will be denoted by $\cp_{prep}$ (respectively, $\cp_{prei}$).  The regular part consists of a family of homogeneous tubes (i.e. tubes of period 1) and a finite number of non-homogeneous tubes, say $\ct_1,\ldots,\ct_l$ respecitvely of periods $r_1,\ldots,r_l$.
We have the following well-known results.


\begin{lemma}\label{l:2.1.1}
\begin{itemize}
\item[(a)] We have $l\leq 3$ and $\sum_{i=1}^l(r_i-1)=|I|-2$.
\item[(b)] Let $P$ be preprojective, $R$ be regular and $I$ be preinjective. Then
\[\Hom(R,P)=\Hom(I,R)=\Hom(I,P)=0,\]
\[\Ext(P,R)=\Ext(R,I)=\Ext(P,I)=0.\]
\item[(c)] Let $R$ and $R'$ be indecomposable regular in different tubes. Then
\[\Hom(R,R')=0=\Ext(R,R').\]
\item[(d)] Let $M\in\ct_i$ for some $1\leqslant i\leqslant l$, and
$$0{\lra}M_2{\lra}M{\lra}M_1{\lra}0$$
be a short exact sequence. Then $M_1\cong I_1\oplus N_1,M_2\cong
P_2\oplus N_2$, where $P_2$ is preprojective, $N_1,N_2\in\ct_i$, and
$I_1$ is preinjective.
\end{itemize}
\end{lemma}

Each tube is an abelian subcategory of $\mod\llz$ and a simple object in a tube is called a \emph{regular simple module} in $\mod\llz$. For a tube of period $r$, there are precisely $r$ simple objects (up to isomorphism), and the sum of dimension vectors of these regular simples is independent of the tube, and this sum will be denoted by $\delta$. Recall that  in Section~\ref{ss:1.1} we have defined the symmetric Euler form $(-,-)$ on $G(\llz)=\bbz I$. In our case, this form is positive semi-definite and its radical is free of rank $1$ generated by $\delta$.  We collect information on the invariants $l$, $r_1,\ldots,r_l$ and $\delta$ of connected tame quivers in the following table.

\[\begin{tabular}{|c|c|c|c|}\hline
type & ~~$\ell$~~ & ~~$r_1,\ldots,r_{\ell}$~~ & $\delta$\\ \hline
\raisebox{-15pt}[0pt]{$\tilde{A}_{n-1,1}$} & \raisebox{-15pt}[0pt]{1} & \raisebox{-15pt}[0pt]{$n-1$} & \xymatrix@!=0.2pc{&&1&&\\ 1 \ar@{-}[urr] \ar@{-}[r]&1\ar@{.}[rr]&&1\ar@{-}[r] &1\ar@{-}[ull] }\\ \hline
\raisebox{-15pt}[0pt]{$\tilde{A}_{p,q}(p,q\geq 2)$} & \raisebox{-15pt}[0pt]{2} & \raisebox{-15pt}[0pt]{$p$, $q$} & \xymatrix@!=0.2pc{&&1&&\\ 1 \ar@{-}[urr] \ar@{-}[r]&1\ar@{.}[rr]&&1\ar@{-}[r] &1\ar@{-}[ull] }\\ \hline
\raisebox{-30pt}[0pt]{$\tilde{D}_n (n\geq 3)$} & \raisebox{-30pt}[0pt]{3} & \raisebox{-30pt}[0pt]{$2,2,n-2$} & \xymatrix@!=0.2pc{1 &&&&&1\\ &2\ar@{-}[ul]\ar@{-}[dl]\ar@{-}[r]&2\ar@{.}[r]&2\ar@{-}[r]&2\ar@{-}[ur]\ar@{-}[dr] &\\ 1 &&&&&1} \\ \hline
\raisebox{-30pt}[0pt]{$\tilde{E}_6$} & \raisebox{-30pt}[0pt]{3} & \raisebox{-30pt}[0pt]{2,3,3} & \xymatrix@!=0.2pc{&&1&&\\ &&2\ar@{-}[u]\ar@{-}[d]&&\\ 1\ar@{-}[r]&2\ar@{-}[r]&3\ar@{-}[r]&2\ar@{-}[r]&1}\\ \hline
\raisebox{-15pt}[0pt]{$\tilde{E}_7$} & \raisebox{-15pt}[0pt]{3} & \raisebox{-15pt}[0pt]{2,3,4} & \xymatrix@!=0.2pc{&&&2\ar@{-}[d]&&&\\ 1\ar@{-}[r]&2\ar@{-}[r]&3\ar@{-}[r]&4\ar@{-}[r]&3\ar@{-}[r]&2\ar@{-}[r]&1}\\ \hline
\raisebox{-15pt}[0pt]{$\tilde{E}_8$} & \raisebox{-15pt}[0pt]{3} & \raisebox{-15pt}[0pt]{2,3,5} & \xymatrix@!=0.2pc{&&3\ar@{-}[d]&&&&&\\ 2\ar@{-}[r]&4\ar@{-}[r]&6\ar@{-}[r]&5\ar@{-}[r]& 4\ar@{-}[r]& 3\ar@{-}[r]& 2\ar@{-}[r]&1}\\ \hline
\end{tabular}\]
\bigskip

Let $K$ be the path algebra over $\bbf_q$ of the Kronecker quiver
$\xymatrix{\cdot\ar@<.5ex>[r]\ar@<-.5ex>[r] & \cdot}$. Then there is
an embedding $\mod K\hookrightarrow\mod\llz$. Precisely, let $e$ be
an extending vertex of $Q$ (i.e. the $e$-th entry of $\delta$ equals
$1$), let $P=P(e)$ be the indecomposable projective $\llz$-module
corresponding to $e$, and let $L$ be the unique indecomposable
preprojective $\llz$-module with dimension vector $\delta+\udim P$.
Let $\fk{C}(P,L)$ be the smallest full subcategory of $\mod\llz$
which contains $P$ and $L$ and is closed under taking extensions,
kernels of epimorphisms, and cokernels of monomorphisms in the
category of $\llz$-modules. Then $\fk{C}(P,L)$ is equivalent to
$\mod K$ (see for example~\cite{LXZ} Section 6.1). This embedding is
essentially independent of the field $\bbf_q$.


\section{Rational Ringel--Hall
algebras}\label{s:singular-hall-alg}

Let $Q$ be a tame quiver with vertex set $I$, and $\Lambda=\bbf_q Q$ be the path algebra of
$Q$ over the finite field $\bbf_q$. In this section we will define the rational Ringel--Hall
algebras of $\Lambda$ and give a description in terms of a set of generators and relations.

\subsection{The rational Ringel--Hall algebra $\ch^{r,x_1,x_2,\cdots,x_t}(\llz)$}
Let $\ct_1,\ldots,\ct_l$ be the non-homogeneous tubes of $\mod\llz$ and assume that they are
respectively of period $r_1,\ldots,r_l$ (see Section~\ref{ss:tame-quiver}), and let $\mathscr{H}_{x_i}$ be the homogeneous tube corresponding to $x_i$, where $x_i$ are $\bbf_q$ rational
points in $ \mathbb{P}^1.$

Let $\ch^*(\llz)$ be the twisted Ringel--Hall algebra of $\Lambda$,
which has a basis $\{u_{[M]}|M\in \mod\Lambda\}$ with structure
constants given by Hall numbers, see Section~\ref{ss:hall-alg}. The
\emph{rational Ringel--Hall algebra} of $\llz$, denoted by
$\ch^{r,x_1,x_2,\cdots,x_t}(\llz)$, is defined as the subalgebra of $\ch^*(\llz)$
generated by $\{u_i,u_{[M]},
u_{[N]}: i\in I, M\in\ct_j, N\in\oplus_{i=1}^t\mathscr{H}_{x_i}
,1\leq j\leq l, ,t\in\bbn\}$, where $x_i$ are $\bbf_q$ rational
points in $ \mathbb{P}^1.$ Later we will prove the existence of Hall
polynomials, so that we have a generic version of the rational
Ringel--Hall algebra $\ch^{r,x_1,x_2,\cdots,x_t}(\llz)$. We now set $\mathscr{D}^{r,x_1,x_2,\cdots,x_t}(\llz)$
to be the subalgebra of $\mathscr{D}(\llz)$ generated by
$\{u_i^{\pm},u_{[M]}^{\pm}, u_{[N]}^{\pm}, K_\mu:i\in I, M \in \ct_j, N\in\oplus_{i=1}^t\mathscr{H}_{x_i}, 1\leqslant
j\leqslant l, \mu\in\bbn[I], \{x_1,x_2,\cdots,x_t\}\subseteq\mathbb{P}^1(\bbf_q) \}.$ Namely, it is the reduced Drinfeld Double of
$\ch^{r,x_1,x_2,\cdots,x_t}(\llz)$.

The
\emph{rational Ringel--Hall algebra} of $\llz$, denoted by
$\ch^r(\llz)$, is defined as the subalgebra of $\ch^*(\llz)$
generated by $\{u_i,u_{[M]}:i\in I, M \in \ct_j, 1\leqslant
j\leqslant l\}$. In the following, we only consider the rational Ringel-Hall algebra $\ch^r(\llz)$. All statements hold for the other rational Ringel-Hall algebras.

\begin{lemma}\label{l:2.1.2}
The $\mathscr{D}^r(\llz)$ is a Hopf algebra over $\bbq(v)$.
\end{lemma}

\begin{proof} We prove that $\mathscr{D}^r(\llz)$ is a Hopf subalgebra of $\mathscr{D}(\llz)$, i.e.
$\mathscr{D}^r(\llz)$ is closed under comultiplication and is closed
under antipode $S$. The former statement follows easily from
Lemma~\ref{l:2.1.1}. For the latter, it is sufficient to check on the
generators of $\mathscr{D}^r(\llz)$, since $S$ is an algebra
homomorphism. That $S(u_i^{\pm}),i\in I$, and
$S(K_\mu),\mu\in\mathbb{N}[I]$ belongs to $\mathscr{D}^r(\llz)$
follows immediately from the definition of $S$. Let $M\in \ct_j,
1\leq j\leq l$. We will prove that
$S(u^+_{[M]})\in\mathscr{D}^r(\llz),$ for $M\in\ct_i, 1\leqslant
i\leqslant l$ by induction on $\udim M$. The proof for $u^-$ is the
same.

Applying the equality $\mu(S \otimes
1)\bigtriangleup=\eta\epsilon$ to $u^+_{[M]}$, we obtain
\begin{eqnarray}\label{f:2.1.1}
S(u^+_{[M]})&=&-S(K_{\udim M})u^+_{[M]}\\
&&\mbox{}-\sum_{M_1,M_2\neq 0}v^{\lr{\udim M_1,\udim M_2}}
\frac{a_{M_1}a_{M_1}}{a_{M}}g^M_{M_1M_2}S(u^+_{[M_1]}K_{\udim
M_2})u^+_{[M_2]}.\nonumber\end{eqnarray} Assume $M_1\cong I_1\oplus
N_1,M_2\cong P_2\oplus N_2$, as in Lemma~\ref{l:2.1.1}. Note that
$\mathrm{Ext}^1(N_1,I_1)=0$. Therefore, we have
$u^+_{[N_1]}u^+_{[I_1]}=v^{\lr{\udim N_1 ,\udim I_1}}u^+_{[M_1]}$.
By induction hypothesis, we have
$S(u^+_{[N_1]})\in\mathscr{D}^r(\llz)$.  By Lemma~6.1 and 6.2 in
\cite{LXZ}, we have  $S(u^+_{[I_1]})\in\mathscr{D}^r(\llz)$. So
$S(u^+_{[M_1]})\in\mathscr{D}^r(\llz)$. Similarly,
$S(u^+_{[M_2]})\in\mathscr{D}^r(\llz)$. Therefore $(\ref{f:2.1.1})$
implies that $S(u^+_{[M]})\in\mathscr{D}^r(\llz)$.
\end{proof}

\subsection{Decomposition of $\ch^{r}(\llz)$}\label{ss:decomposition} In this subsection, we
follow an idea of Sevenhant and Van den Bergh to obtain subalgebras
of $\ch^r(\llz)$ and $\mathscr{D}^r(\llz).$ (See also \cite{HX}.)

The twisted Ringel--Hall algebra $\ch^*(\Lambda)$ is naturally
$\mathbb{N}[I]$-graded, and so are its subalgebras $\ch^r(\llz)$ and
$\cc(\llz)$. We define a partial order on $\mathbb{N}[I]$: for
$\alpha,\beta\in\mathbb{N}[I]$, $\alpha\leq\beta$ if and only if
$\beta-\alpha\in\mathbb{N}[I]$. Clearly,
$\mathcal{C}(\llz)_\bz=\ch^r(\llz)_\bz$ if $\bz<\dz$.

Recall from Section~\ref{ss:double-hall-alg} that
$\varphi:\ch^+(\llz)\times\ch^-(\llz){\lra}\bbq(v)$ is a
non-degenerate bilinear form. It is easy to see that the restriction
of $\varphi$ on $\ch^{r,+}_\az\times\ch^{r,-}_\az$ is also
non-degenerate for all $\az\in \bbn[I]$. We now define
 $$\mathcal {L}^\pm_\dz=\{x^\pm\in\ch^{r,\pm}_{\dz}(\llz)|~\varphi(x^\pm,\mathcal {C}^{\mp}(\llz))=0\}.$$
 The non-degeneracy of $\varphi$ implies
\[\ch^{r,\pm}(\llz)_\dz=\mathcal {C}^{\pm}(\llz)_\dz\oplus\mathcal
{L}^\pm_\dz.\]
Let $\mathscr{D}^r(1)$ be the subalgebra of $\mathscr{D}^r$
generated by $\mathcal {C}^{\pm}(\llz)$ and $\mathcal{L}^{\pm}_\dz.$
Then we have a triangular decomposition
$$\mathscr{D}^r(1)=\mathscr{D}^r(1)^-\otimes\mathcal
{T}\otimes\mathscr{D}^r(1)^+.$$

Suppose $\mathcal {L}^{\pm}_{(m-1)\dz}$ and
$\mathscr{D}^r(m-1)^{\pm}$ have been defined, we inductively define
$\mathcal {L}^{\pm}_{m\dz}$ as follows:
\[
\mathcal{L}^{\pm}_{m\dz}=\{x^{+}\in\ch^{r,\pm}_{m\dz}|~\varphi(x^{\pm},\mathscr{D}^r(m-1)^\mp)=0\}.
\]
Let $\mathscr{D}^r(m)$ be the subalgebra of $\mathscr{D}^r$
generated by $\mathscr{D}^r(m-1)^{\pm}$ and $\mathcal{L}^{\pm}_{m\dz}.$ As
in the $m=1$ case, we have a triangular decomposition
$$\mathscr{D}^r(m)=\mathscr{D}^r(m)^-\otimes\mathcal
{T}\otimes\mathscr{D}^r(m)^+.$$
In this way we obtain a chain of subalgebras of $\mathscr{D}^r$
\[\cc(\llz)\subset \mathscr{D}^r(1)\subset\mathscr{D}^r(2)\subset\ldots\subset\mathscr{D}^r(m)\subset\ldots\subset\mathscr{D}^r\]
such that $\mathscr{D}^r=\bigcup_{m\in\bbn}\mathscr{D}^r(m)$.


\begin{lemma}\label{l:2.2.1}
Let $\eta_{n\dz}=\dim_{\bbq(v)}\mathcal
{L}^+_{n\dz}=\dim_{\bbq(v)}\mathcal {L}^-_{n\dz}$.  Then
$\eta_{n\dz}=l$.
\end{lemma}

\begin{proof} By proposition~6.5 in~\cite{LXZ},
we know that $\mathcal {C}^+(\llz)/(v-1)$ has a basis consiting of the following elements
\begin{itemize}
\item[a)] $\Psi(u_{[M(\az)]})$ for
 $\az\in \Phi^+_{Prep}$;
\item[b)] $\Psi(u_{\az,i})$ for $\az\in\ct_i, $ the  real
 roots,  $i=1,\ldots,l;$
\item[c)] $\Psi(u_{j,m\dz,i}-u_{j+1,m\dz,i}),$ $m\geq 1,$
 $1\leq j\leq r_i-1,$ $i=1,\ldots,l;$
\item[d)] $\Psi(\te_{n\dz}), n\geq 1$
\item[e)] $\Psi(u_{[M(\bz)]})$ for $\bz\in \Phi^+_{Prei}$.
\end{itemize}
While according to the definition $\ch^r$ the space  $\ch^r/(v-1)$ has a basis consiting of elements in
a) b) d) e) and all $\Psi(u_{j,m\dz,i}),$ $m\geq 1,$
 $1\leq j\leq r_i,$ $i=1,\ldots,l$. Thus we have $\eta_{n\dz}=l$ by the
 construction of $\mathcal
{L}^+_{n\dz}$ for all $n.$ 
\end{proof}

For each $n\dz,$ there exists a basis $\{x_n^1,\ldots,x_n^l\}$ of
$\mathcal {L}^+_{n\dz}$ and a basis $\{y_n^1,\ldots,y_n^l\}$ of
$\mathcal {L}^-_{n\dz}$ such that
$$\varphi(x_n^p,y_n^q)=\frac{1}{v-v^{-1}}\dz_{pq}.$$
Here $\dz_{pq}$ denotes the Kronecker's delta. Then we have
$$x_n^py_m^q-y_m^qx_n^p=\frac{K_{n\dz}-K_{-n\dz}}{v-v^{-1}}\dz_{pq}\dz_{mn},$$
for all $m,n\in\mathbb{N},1\leqslant p,q\leqslant l$
(see~\cite{HX}).

Let $J=\{(n\dz,p):n\in\mathbb{N},1\leqslant p\leqslant l\}.$ Define
\begin{eqnarray*}
\theta_i&=&\begin{cases} i & \text{ if } i\in I,\\ n\dz & \text{ if
} i=(n\dz,p)\in J,\end{cases}\\
x_i&=&\begin{cases} u_i^+ & \text{ if } i\in I,\\ x_n^p & \text{ if
}
i=(n\dz,p)\in J,\end{cases}\\
y_i&=&\begin{cases} -v^{-1}u_i^- & \text{ if } i\in I,\\-v^{-1}y_n^p
& \text{ if } i=(n\dz,p)\in J.\end{cases}\end{eqnarray*} By a
theorem of Sevenhant and Van den Bergh, we obtain that
$\mathscr{D}^r$ is generated by
$$\{x_i,y_i|~i\in I\cup J\}\cup\{K_\mu:\mu\in\bbn[I]\}$$
with the defining relations
\begin{eqnarray}
&K_0 =1,K_\mu K_\nu=K_{\mu+\nu} \text{ for all } \mu,\nu\in\bbn[I]&
\label{f:2.2.1}\\
&K_\mu x_i = v^{(\mu,\theta_i)}x_i K_\mu, \text{ and } K_\mu y_i =
v^{-(\mu,\theta_i)}y_i K_\mu \text{ for all } i\in I\cup J , \mu\in
\bbn[I]&\label{f:2.2.2}\\
&x_iy_j-y_jx_i=\frac{K_{\theta_i}-K_{-\theta_i}}{v-v^{-1}}\dz_{ij}
\text{ for all }i,j\in I\cup J& \label{f:2.2.3}\\
&\sum_{p+p'=1-a_{ij}}(-1)^px_i^{(p)}x_jx_i^{(p')}=0 \text{ and}
\sum_{p+p'=1-a_{ij}}(-1)^py_i^{(p)}y_jy_i^{(p')}=0& \label{f:2.2.4}\\
&x_ix_j =x_jx_i \text{ and } y_iy_j=y_jy_i \text{  for all } i,j\in
I\cup J \text{ with } (\theta_i,\theta_i)=0.& \label{f:2.2.5}
\end{eqnarray}
Here for an element $x$ and a positive integer $p$ the symbol $x^{(p)}$ denotes the divided power $\frac{x^p}{[p]!}$, where  $[p]=\frac{v^p-v^{-p}}{v-v^{-1}}$ and $[p]!=[1][2]\cdots[p]$.


Applying the relations above, the next statement is clear.

\begin{proposition}\label{p:2.2.2}
\begin{itemize}
\item[(a)] The elements $\{x_n^i|~n\in\bbn, 1\leqslant i\leqslant l\}$ are central in $\ch^{r,+}$.
Dually, the elments $\{y_n^i|~n\in\bbn, 1\leqslant i\leqslant l\}$ are central in $\ch^{r,-}$.
\item[(b)] We have an algebra isomorphism $\ch^{r,+}\cong \mathcal {C}^+\otimes \bbq(v)[x_n^i|~n\in\bbn, 1\leqslant i\leqslant l$.
\item[(c)] The element $x_n^j$ commutes with $\ch^{r,-}$, and the element $y_n^j$ commutes with $\ch^{r,+}$.
\end{itemize}
\end{proposition}


In particular, when the quiver $Q$ is of type
$\widetilde{A}_{n-1,1}$ for some $n\geq 3$, we have $l=1$ and we
have an algebra isomorphism $\ch^{r,+}\cong
U_q(\widetilde{sl}_n)^+\otimes \bbq(v)[x_1,\ldots,x_n,\ldots]$. This
means that $\ch^{r,+}$ is isomorphic to the $q$-Fock space of type
$A$.  Also this proposition implies that the algebra structure
$\ch^r$ not only  depends on the type, but also on the orientation.




\section{The quantum extended Kac--Moody algebra}

\subsection{} Let $(I,(,))$ be a datum in the sense of Green [G,
3.1]. Let $J$ be an index set and let
$\theta:J\rightarrow\bbn[I]\backslash\{0\},j\mapsto \theta_j$ be a
map with finite fibers. We denote by $\widetilde{G}$ the triple $(I,
(,),\{\theta_j: j\in J\})$, and call it an \emph{extended Green
datum}. We extend $\theta$ to a map $\theta:I\cup J\rightarrow\bbn
[I]\backslash\{0\}$ by setting $\theta_i=i$ for $i\in I$.

Let $Q$ be a tame quiver with vertex set $I$, and let $\llz$ be the path algebra. Recall that a Hopf subalgebra $\mathscr{D}^r(\llz)$ of $\mathscr{D}(\llz)$ is defined in Section~\ref{s:singular-hall-alg}.
Let $\widetilde{G}$ be the extended Green datum
corresponding to $\mathscr{D}^r(\Lambda)$. Precisely, $J=\{(n\dz,p):1\leqslant p\leqslant l\}$ and
$\theta_{(n\delta,p)}=n\delta$. From $\widetilde{G}$ we define a new
datum $\widetilde{G}'=(I\cup J, (,)'),$ where
$(i,j)'=(\theta_i,\theta_j)$ for all $i,j\in I\cup J.$ Let
$\mathscr{D}'(\Lambda)$ be the reduced Drinfeld double corresponding
to $\widetilde{G}'$. We have the following proposition similar to~\cite[Proposition 3.8]{DX}.

\begin{proposition}\label{p:3.1} There exists a surjective Hopf algebra
homomorphism $F:\mathscr{D}'(\Lambda)\rightarrow
 \mathscr{D}^{r}(\Lambda)$ such that $F(x_i^{\pm})=x_i^{\pm}$ and
 $F(K_i)=K_{\theta_i}$ for $i\in I\cup J,$ and $\ker F$ is the ideal
 generated by  $\{K_j-K_{\theta_j}| j\in J\}.$
 \end{proposition}

In particular, we have ${\mathscr{D}'}^{> 0}\cong \ch^{r,+}. $

\subsection{} Let
$\mathbb{U}=\mathbb{U}^-\otimes\mathbb{U}^0\otimes\mathbb{U}^+$ be
the quantized enveloping algebra in the sense of Drinfeld and Jimbo
with the generators $\{E_i,F_i,K_i,K_{-i}| i\in I\cup J\}$ subject
to relations similar to (\ref{f:2.2.1})--(\ref{f:2.2.5}) as in Section~\ref{ss:decomposition}. $\mathbb{U}$
is called the \emph{quantum extended Kac--Moody algebra}. Similar to~\cite[Corollary 8.3]{DX}, we have
$$\mathbb{U}\cong \mathscr{D}'.$$ So we obtain

\begin{proposition}\label{p:3.2}
 There exists an  algebra
 isomorphism $G:\mathbb{U}^+\stackrel{\sim}{\rightarrow} \ch^{r,+}$ such that $G(E_i)=x_i^{+}$
 for $i\in I\cup J.$
\end{proposition}




\section{PBW-basis of $\mathbb{U}^+$}\label{s:4}

\subsection{}\label{ss:4.1}
Let $Q$ be a connected tame quiver without oriented cycles with vertex set $I$, and let $\llz=\bbf_q Q$ be the path algebra of $Q$ over the finite field $\bbf_q$. Recall that the regular part of $\mod\llz$ is the direct sum of a family of homogeneous tubes and finitely many non-homogeneous tubes $\ct_1,\ldots,\ct_l$.

\begin{lemma}\label{l:regular-commutative}
 Let $M$ be a module in the direct sum of the homogeneous tubes, and $N$ be a regular module. Then $u_{[M]}*u_{[N]}=u_{[N]}*u_{[M]}$. If in addition no direct summands of $N$ belongs to the same tube as any indecomposable dirct summand of $M$, then $u_{[M]}*u_{[N]}=u_{[N]}*u_{[M]}=u_{[M\oplus N]}$.
\end{lemma}
\begin{proof}
 If no direct summands of $N$ belongs to the same tube as any indecomposable dirct summand of $M$, the statement follows from Lemma~\ref{l:2.1.1} (c) and the fact that the dimension vector of $M$ is a multiple of $\delta$, which lies in the radical of the Euler form. If $N$ is also in the direct sum of the homogeneous tubes, the statement is proved in~\cite{Zh}. Generally we write $u_{[N]}=u_{[N_1]}*u_{[N_2]}$ with $N_1\oplus N_2\cong N$ such that $N_1$ belongs to the direct sum of the homogeneous tubes and $N_2$ belongs to the direct sum of the non-homogeneous tubes. Then
\begin{eqnarray*}
u_{[M]}*u_{[N]}&=&u_{[M]}*u_{[N_1]}*u_{[N_2]}\hspace{7pt}=\hspace{7pt}u_{[N_1]}*u_{[M]}*u_{[N_2]}\\
&=&u_{[N_1]}*u_{[N_2]}*u_{[M]}\hspace{7pt}=\hspace{7pt}u_{[N]}*u_{[M]}.
\end{eqnarray*}
\end{proof}

Let $P$, $L$, $K$ and $\fk{C}(P,L)$ be as in Section~\ref{ss:tame-quiver}, and let $F:\mod K\cong\fk{C}(P,L)\hookrightarrow\mod\llz$ be the exact embedding.
It gives rise to an injective homomorphism of algebras, still denoted
by $F:\ch^*(K)\hookrightarrow\ch^*(\llz).$ In $\ch^*(K)$ a distinguished element $E_{m\dz_K}$ is defined for any $m\geq 1$, and we set $E_{m\dz}=F(E_{m\dz_K}),$ see~\cite{LXZ} for more details. Since $E_{m\dz_K}\in
\cc^*(K),$ and $\lr{L}, \lr{P}\in \cc^*(\llz)$, it follows that $E_{m\dz}$ is in
$\cc^*(\llz)$ and even  in $\cc^*(\llz)_{\cz}$, where $\cz=\bbq[v,v^{-1}]$.  Let $\cal K$ be the
subalgebra of $\cc^*(\llz)$ generated by $E_{m\dz}$ for $m\in\bbn,$
it is a polynomial ring on infinitely many variables
$\{E_{m\dz}|m\geq1\},$ and its integral form is the polynomial ring
on variables $\{E_{m\dz}|m\geq1\}$ over $\cz.$

We denote by $\fk{C}_0$ (respectively, $\fk{C}_1$) the full subcategory of
$\fk{C}(P,L)$ consisting of the $\Lambda$-modules which belong to
homogeneous (respectively, non-homogeneous) tubes of $\mod\Lambda.$

We now decompose $E_{n\dz}$ as
   $$ E_{n\delta}=E_{n\delta,1}+E_{n\delta,2}+E_{n\delta,3} , $$
 where
\begin{eqnarray}
   E_{n\delta,1}&=&v^{-n|\delta|}
   \sum_{[M]:M\in\fk{C}_1,\udim M=n\delta}u_{[M]}\label{f:4.1.1}\\
    E_{n\delta,2}&=&v^{-n|\delta|}
   \sum_{\substack{[M]:M\in\fk{C},\udim M=n\delta\\
   M=M_1\oplus M_2, 0\neq M_1\in \fk{C}_1, 0\neq M_2\in
   \fk{C}_0}}u_{[M]}\label{f:4.1.2}\\
   E_{n\delta,3}&=&
   v^{-n|\delta|}\sum_{[M]:M\in\fk{C}_0, \udim M=n\delta}u_{[M]},\label{f:4.1.3}
   \end{eqnarray}
where $|\delta|$ is the sum of all entries of $\delta$.

\begin{lemma}\label{l:E-n-delta} Let $n,n'$ be two positive integers. Then we have
\begin{itemize}
 \item[(a)] $E_{n\delta,1}*E_{n'\delta,3}=E_{n'\delta,3}*E_{n\delta,1}$;
 \item[(b)] $E_{n\delta,2}=\sum_{m=1}^{n-1} E_{m\delta,1}*E_{(n-m)\delta,3}$;
 \item[(c)] $E_{n\delta,3}*E_{n'\delta,3}=E_{n'\delta,3}*E_{n\delta,3}$.
\end{itemize}
\end{lemma}
\begin{proof} (a) and (c) follows from Lemma~\ref{l:regular-commutative}. (b) holds because
 \begin{eqnarray*}
E_{n\delta,2}
&=&v^{-n|\delta|}
   \sum_{\substack{[M_1\oplus M_2]:\udim (M_1\oplus M_2)=n\delta,\\
   0\neq M_1\in \fk{C}_1, 0\neq M_2\in
   \fk{C}_0}}u_{[M_1\oplus M_2]}\\
&=&v^{-n|\delta|}
   \sum_{\substack{[M_1\oplus M_2]:\udim (M_1\oplus M_2)=n\delta,\\
   0\neq M_1\in \fk{C}_1, 0\neq M_2\in
   \fk{C}_0}}u_{[M_1]}*u_{[ M_2]}\\
&=&v^{-n|\delta|}\sum_{m=1}^{n-1}
   \sum_{\substack{[M_1]:M_1\in\fk{C}_1,\\ \udim M_1=m\delta}}
   \sum_{\substack{[M_2]:M_2\in\fk{C}_0,\\ \udim M_2=(n-m)\delta}}u_{[M_1]}*u_{[ M_2]}\\
&=&\sum_{m=1}^{n-1}
   (\sum_{\substack{[M_1]:M_1\in\fk{C}_1,\\ \udim M_1=m\delta}}v^{-m|\delta|}u_{[M_1]})
   *(\sum_{\substack{[M_2]:M_2\in\fk{C}_0,\\ \udim M_2=(n-m)\delta}}v^{-(n-m)|\delta|}u_{[M_2]})\\
&=&\sum_{m=1}^{n-1} E_{m\delta,1}*E_{(n-m)\delta,3}.
\end{eqnarray*}
\end{proof}

For a partition
$\textbf{w}=(w_1,\ldots,w_t)$ of $n,$ we define
\begin{eqnarray*}
E_{\textbf{w}\dz}&=&E_{w_1\dz}*\cdots*E_{w_t\dz}\\
E_{\textbf{w}\dz,3}&=&E_{w_1\dz,3}*\cdots*E_{w_t\dz,3}.\end{eqnarray*}
Let ${\bf P}(n)$ be the set of all partitions of $n,$ and
$\lr{N}=v^{-\dim N+\dim \ed(N)}u_{[N]}.$ Set
$${\bf
B}=\{\lr{P}*\lr{M}*E_{\textbf{w}\dz,3}*\lr{I}|~ P\in\cp_{prep},
M\in\bigoplus_{i=1}^l\ct_i, I\in\cp_{prei},{\bf w}\in{\bf P}(n),
n\in\bbn\},$$
where recall that $\cp_{prep}$ respectively $\cp_{prei}$ is the set of isomorphism classes of preprojective respectively preinjective $\llz$-modules, and $\ct_1,\ldots,\ct_l$ are the non-homogeneous tubes of $\mod\llz$.
Then we have the following:


\begin{theorem}\label{t:4.1.1}  The set ${\bf B}$ is a
$\bbq(v)-$basis of $\ch^{r,+}$ ($\cong\mathbb{U}^+$).
\end{theorem}

\begin{proof}
We have by definition that $E_{n\delta}$ and $E_{n\delta,1}$ belong
to $\ch^{r,+}$. Then it follows from Lemma~\ref{l:E-n-delta} (b) by
induction on $n$ that both $E_{n\delta,2}$ and $E_{n\delta,3}$
belong to $\ch^{r,+}$. As a consequence, the set ${\bf B}$ is contained in
$\ch^{r,+}$. Because ${\bf B}$ is linear independent over $\bbq(v)$,
it remains to show that ${\bf B}$ linearly spans $\ch^{r,+}$.

Let $\Pi_i^a$ be the set of aperiodic $r_i$-tuples
of partitions, for all $1\leqslant i\leqslant l.$ Set
\begin{eqnarray*}
  {\bf B}^c&=&\{\lr{P}\ast E_{\pi_1
}\ast\cdots\ast  E_{\pi_l }\ast E_{{\bf
w}\dz}\ast\lr{I}:\\
&&P\in\cp_{prep},I\in\cp_{prei},\pi_i\in\Pi_i^a,1\leqslant i\leqslant l, {\bf
w}\in{\bf P}(n),n\in\bbn\}.
 \end{eqnarray*}
By  Proposition~7.2 in \cite{LXZ}, we know that ${\bf B}^c$
 is a
$\bbq(v)$-basis of $\cc^*(\llz).$  By definition $\ch^{r,+}$ is
generated by ${\bf B}^c$ and $\{u_{[M]}:M\in\bigoplus_{i=1}^l
\ct_i\}$. The latter elements belong to ${\bf B}$, so now it remains
to prove that each element in ${\bf B}^c$ ,$u_{[M]}\ast\lr{P}$ and
$(E_{{\bf w}\dz}\ast\lr{I})\ast u_{[M]}$ can be linearly spanned by
${\bf B}$.

For ${\bf w}=(w_1,\ldots,w_t)\in {\bf P}(n)$, we have by Lemma~\ref{l:E-n-delta} (b)
\begin{eqnarray*}
E_{{\bf w}\delta}&=&E_{w_1\delta}*\cdots*E_{w_t\delta}\\
&=&\prod_{j=1}^t(E_{w_j\delta,1}+\sum_{m_j=1}^{w_j-1}E_{m_j\delta,1}*E_{(w_j-m_j)\delta,3}+E_{w_j\delta,3})
\end{eqnarray*}
By Lemma~\ref{l:E-n-delta} (a) (c), we can write $E_{{\bf w}\delta}$ as a linear combination of elements of the form
$E_{m_1\delta,1}*\cdots *E_{m_r\delta,1}*E_{m'_1\delta,3}*\cdots*E_{m'_{r'}\delta,3}$,
which itself is a linear combination of elements of the form $\langle M\rangle*E_{{\bf w'}\delta,3}$,
where $M\in\bigoplus_{i=1}^l \ct_i$ and ${\bf w'}$ is a partition.

Similarly, we can prove that each element of form
$u_{[M]}\ast\lr{P}$ and $(E_{{\bf w}\dz}\ast\lr{I})\ast u_{[M]}$ can
be linearly spanned by ${\bf B}$. This completes the proof.
\end{proof}

\section{ Hall polynomials of affine type }

\bigskip

\subsection{} In \cite{H}, A.Hubery has proved the existence of Hall
polynomials for tame quivers for Segre classes by using comultiplication. In this subsection,
by using the rational Ringel-Hall algebras, we
give a simple and direct proof for the existence of Hall polynomials
for tame quivers.

Let $\ch^{r,x_1,x_2,\cdots,x_t}$ be a rational Ringel-Hall algebra as in section 2.1.
We now give a new decomposition of $E_{n\dz}$ as follows

   $$ E_{n\delta}=E_{n\delta,1}+E_{n\delta,2}+E_{n\delta,3} , $$

 where
 \begin{eqnarray}\label{f:8.1.1}
   E_{n\delta,1}=v^{-n \dim S_1-n \dim S_2}
   \sum_{[M],M\in\fk{C}_1\bigoplus\oplus_{x_t}\mathscr{H}_{x_t},\udim
   M=n\delta}u_{[M]},
     {\qquad\qquad\qquad\quad}\end{eqnarray}
 \begin{eqnarray}\label{f:8.1.2}
    E_{n\delta,2}=v^{-n \dim S_1-n \dim S_2}
   \sum_{\substack{[M],\udim M=n\delta,\\
   M=M_1\oplus M_2, 0\neq M_1\in \fk{C}_1\bigoplus\oplus_{x_t}\mathscr{H}_{x_t}, 0\neq M_2\in
   \fk{C}_0\backslash\oplus_{x_t}\mathscr{H}_{x_t}}}u_{[M]},\end{eqnarray}
 \begin{eqnarray}\label{f:8.1.3}
   E_{n\delta,3}=
   v^{-n \dim S_1-n \dim S_2}\sum_{[M], M\in\fk{C}_0\backslash\oplus_{x_t}\mathscr{H}_{x_t},
   \udim M=n\delta}u_{[M]}. {\qquad\qquad\qquad\quad\quad}
  \end{eqnarray}

  Let
$\textbf{w}=(w_1,\cdots,w_t)$ be a partition of $n,$ we then define
$$E_{\textbf{w}\dz,3}=E_{w_1\dz,3}*\cdots*E_{w_t\dz,3}.$$

Let ${\bf P}(n)$ be the set of all partitions of $n,$ and
$\lr{N}=v^{-dim N+dim End(N)}u_{[N]}.$ Set
$${\bf
B}'=\{\lr{P}*\lr{M}*E_{\textbf{w}\dz,3}*\lr{I}|~ P\in\cp_{prep},
M\in\oplus_{i=1}^l\ct_i\bigoplus\oplus_{x_t}\mathscr{H}_{x_t},
I\in\cp_{prei},{\bf w}\in{\bf P}(n), n\in\bbn\}.$$

Similar to Theorem ~4.1.1,   we have the following:

\begin{proposition}\label{p:8.1.1 } {\sl  The set ${\bf B}'$ is a
$\bbq(v)-$ basis of $\ch^{r,x_1,x_2,\cdots,x_t}$.} \end{proposition}

\begin{theorem}\label{t:8.1.2}  Let $Q$ be a affine quiver ,  let
$P_i$ (resp. $R_i,I_i$) be a pre-projective (resp. nonhomogeneous
regular, pre-injective) $\bbf_q Q-$module, and let $H_i\in
\oplus_{j=1}^{t}\mathscr{H}_{x_j}$ be a homogeneous regular $\bbf_q
Q-$module with $x_j$ being $\bbf_q-$ rational point in
$\mathbb{P}^1$ for $i=1,2,3; t\in \bbn.$ Let $X_i=P_i\oplus
R_i\oplus H_i\oplus I_i,i=1,2,3.$

Then there exists a Hall polynomial
$\varphi_{X_1X_2}^{X_3}(x)\in\bbq[x]$ such that
$$\varphi_{X_1X_2}^{X_3}(q)=g_{X_1X_2}^{X_3}.$$ \end{theorem}

\begin{proof} Since
$$\lr{P_1\oplus M_1\oplus I_1}*\lr{P_2\oplus M_2\oplus I_2}=
a_{12}^3(v)\lr{P_3\oplus M_3\oplus I_3}+\text{ other terms},$$ and
$a_{12}^3(v)\in \bbq (v),$ then we have
$$\begin{array}{l}g_{P_1\oplus R_1\oplus H_1\oplus I_1,P_2\oplus
R_2\oplus H_2\oplus I_2}^{P_3\oplus R_3\oplus
H_3\oplus I_3} \\
=v^{dim_{\bbf_q}End(P_3\oplus M_3\oplus
I_3)-dim_{\bbf_q}End(P_1\oplus M_1\oplus
I_1)-dim_{\bbf_q}End(P_2\oplus M_2\oplus I_2)-\lr{\udim P_1\oplus
R_1\oplus H_1\oplus I_1, \udim P_2\oplus R_2\oplus H_2\oplus I_2}
}a_{12}^3(v).\end{array}$$

Set

$$(*)\begin{array}{l}\varphi_{X_1X_2}^{X_3}(v^2)\\
=v^{dim_{\bbf_q}End(P_3\oplus M_3\oplus
I_3)-dim_{\bbf_q}End(P_1\oplus M_1\oplus
I_1)-dim_{\bbf_q}End(P_2\oplus M_2\oplus I_2)-\lr{\udim P_1\oplus
R_1\oplus H_1\oplus I_1, \udim P_2\oplus R_2\oplus H_2\oplus I_2}
}a_{12}^3(v).\end{array}$$

On the other hand, we know that
$$v^{dim_{\bbf_q}End(P_3\oplus M_3\oplus
I_3)-dim_{\bbf_q}End(P_1\oplus M_1\oplus
I_1)-dim_{\bbf_q}End(P_2\oplus M_2\oplus I_2)-\lr{\udim P_1\oplus
R_1\oplus H_1\oplus I_1, \udim P_2\oplus R_2\oplus H_2\oplus I_2}
}a_{12}^3(v)$$  takes the positive integer value while $v^2$ takes
infinite many positive integer values. Then
$\varphi_{X_1X_2}^{X_3}(v^2)$ is a polynomial of $v^2$ over $\bbq.$
Thus the proof is complete.\end{proof}

\bigskip

\section{Canonical bases of $\mathbb{U}^+$ ($\cong\ch^r(\llz)$)}

In this section we give the main theorem of this paper, that is, a description of the canonical basis of $\ch^r(\llz)\cong\mathbb{U}^+$.

\subsection{}
Let $Q$ be a tame quiver with vertex set $I$, and $\Lambda=\bbf_q Q$
be the path algebra of $Q$ over the finite field $\bbf_q$. We denote
by $M(x),x\in\bbe_\az,$ the $\llz$-module of dimension vector $\az$
corresponding to $x.$ For subsets $\ca\subset\bbe_{\az}$ and
$\cb\subset\bbe_{\bz},$ we define the extension set $\ca\star\cb$ of
$\ca$ by $\cb$  to be
\begin{eqnarray*} \ca\star\cb&=&\{z\in\bbe_{\az+\bz}|\ \text{there exists an exact
sequence}\ \\
&& \quad 0\ra M(x)\ra M(z)\ra M(y)\ra 0 \ \text{with}\
x\in\cb,\ y\in\ca\}.
\end{eqnarray*}
It follows from the definition that
$\ca\star\cb=p_3p_2(p_1^{-1}(\ca\times \cb))$, see
Section~\ref{ss:construction-lusztig} for the definitions of
$p_1,p_2$ and $p_3$. Because $p_1$ is a locally trivial fibration ,
we have $\overline{\ca}\star \overline{\cb} \subseteq
\overline{\ca\star\cb}$(see Lemma~2.3 in \cite{LXZ}) . In
particular, $\ol{\co}_M\star\ol{\co}_N=\ol{\co}_{M\oplus N}$ if
$\ext(M,N)=0,$ i.e. $\co_{M\oplus N}$ is open and dense  in
$\ol{\co}_M\star\ol{\co}_N.$

Set $\cdim\ca=\dim\bbe_{\az}-\dim\ca.$ We will need the following:


\begin{lemma}[\cite{Re}]\label{l:5.1.1}
Let  $\az, \bz\in\bbn I$. If $\ca\subset\bbe_{\az}$ and $\cb\subset\bbe_{\bz}$  are
irreducible algebraic varieties and are stable under the action of
$G_{\az}$ and $G_{\bz}$ respectively, then $\ca\star\cb$ is
irreducible and stable under the action of $G_{\az+\bz},$ too.
Moreover,
$$\cdim \ca\star\cb =\cdim\ca+\cdim\cb-\lr{\bz,\az}+r,$$
where $0\leq r\leq\min\{\dim_k\hom(M(y),M(x))|y\in\cb,x\in\ca\}$.
\end{lemma}

Recall that for $x\in \bbe_\az$ the symbol $\co_x$ (or $\co_{M(x)}$) denotes the $G_\az$-orbits of $x$.
We now introduce two orders in $\llz-$mod as follows:
\begin{itemize}
\item $N\leq_{deg} M$  if $\co_N\subseteq\overline{\co}_M$.
\item $N\leq_{ext} M$  if there exist $M_i,U_i,V_i$ and short exact sequence
\[0\lra U_i\lra M_i\lra V_i\lra 0\]
such that $M=M_1,M_{i+1}=U_i\oplus V_i, 1\leqslant
i\leqslant p$, and $N=M_{p+1}$ for $p\in\mathbb{N}$.
\end{itemize}
%

For any two modules $M$ and $N$. The module $L$ is called the generic extension of $M$ by $N$ if the orbit $\co_L$ is open and dense in the irreducible variety $\co_M*\co_N,$ and is denoted by $L=M\diamond N.$
Thus $\co_{L'}\subseteq\overline{\co}_L$ for any extension $L'$ of $M$ by $N,$ i.e., $L'\leq_{deg}L.$

\begin{proposition}\label{p:5.1.2}\cite{Z}
 The orders $\leq_{deg},\leq_{ext}$ are equivalent in $\llz-mod.$
\end{proposition}


 Let $X$ be a variety of pure dimension $n$ over $k$ and let $l$ be a prime $\neq p.$ We denote by $\cp_X$ the category of perverse sheaves on an
algebraic variety $X.$  Let $f$ be a locally closed embedding from
$X$ to algebraic variety $Y$. One has the intermediate extension
functor
$$f_{!*}:\cp_X\lra \cp_Y, P\longmapsto \Im\{ ^p\mathcal{H}^0(f_!P)\lra {^p\mathcal{H}}^0(f_*P)\}.$$

Let $V$ be a locally closed, smooth, irreducible subvariety of $X$,
of dimension $d$ and let $\mathcal{L}$ be an irreducible
$\mathbb{Q}_l-$ local system on $V.$ Then $\mathcal{L}[d]$ is an
irreducible perverse sheaf on $V$ and there is a unique irreducible
perverse sheaf $\widetilde{\mathcal{L}}[d],$ whose restriction to
$V$ is $\mathcal{L}[d]$, we have
$\widetilde{\mathcal{L}}[d]=IC_{\overline{V}}(\mathcal{L}),$ where
$IC_{\overline{V}}(\mathcal{L})$ is the intersection cohomolgy
complex of Deligne -Goresky-Macpherson of $\overline{V}$ with
coefficients in $\mathcal{L}.$ The extension of
$\widetilde{\mathcal{L}}[d]$ to $X$ (by $0$ outside $\overline{V}$)
is an irreducible perverse sheaf on $X.$

 In particular, suppose $\mathcal{L}$ is a local system on a
nonsingular Zariski dense open subset $j:U\lra Y$ of an irreducible
algebraic variety $Y$ defined over $\bbf_p,$ and let $IC(Y,\mathcal{L}):=j_{!*}\mathcal{L}.$
We shall say that $Y$ is pure (resp. very pure) if for any $y\in Y(\bbf_{p^r})$
and for any $i,$ all eigenvalues of $(Fr^*)^r$ on $\ch_y^i(IC(Y,\overline{\bbq}_l))$ have complex absolute value $\leq p^{ir/2}$ (resp. $= p^{ir/2}$).

\begin{lemma} \cite{D} If $X$ is projective and pure, then all the eigenvalues of $Fr^*$ on $IH^i(X)$ have complex absolute value $ p^{i/2}$.
\end{lemma}

\begin{lemma} \cite{KL} Let $Y$ be an irreducible closed $\bbf_p-$ subvariety of $k^{\bbn}$ invariant under $\bbc_m-$ action defined by $\lambda (z_1,z_2,...,z_n)=(\lambda^{a_1}z_1,...,\lambda^{a_n}z_n),$ where $a_1>0,...,a_n>0$
Then

(1) $IH^i(Y)=\ch_0^i(IC(Y,\overline{\bbq}_l)),$ for all $i$, where $0$ is the origin of $k^{\bbn}$.

(2) If $Y-0$ is very pure, then $Y$ is very pure.
\end{lemma}

\begin{definition}\label{d:5.1.3} Let $IC(Y,\mathcal{L})$ as above. and let $q=p^r.$ Then $IC(Y,\mathcal{L})$ is said to \emph{strong pure} if all eigenvalues of $(Fr^r)^*$ on the stalks at $x$ of $"i-"\text{th}$ cohomology sheaf of
$IC(Y,\mathcal{L})$ are equal to $p^{ir/2}$, for any $x\in X(\bbf_{p^r})$.
\end{definition}

 \begin{lemma}\label{l:5.1.4}  Let $Y$ be an irreducible algebraic variety of dimension $m$, and let $p:X\rightarrow Y$ be a smooth morphism of relative dimension $d$. Suppose $U_0$  is a nonsingular
Zariski dense open subset of $Y,$ $j_0:U_0\rightarrow Y$ is an open embedding. Then we have the following
cartesian square
\[
 \xymatrix{V=p^{-1}(U_0)\ar[d]^{p|_U}\ar[r]^(0.65)j &  X\ar[d]^p\\
 U_0 \ar[r]^{j_0} &Y.}
\]
If ${(j_0)}_{!*}\overline{\bbq}_l$ is strong pure, then $j_{
!*}\overline{\bbq}_l[d]$ is strong pure.
\end{lemma}

\begin{proof} By the definition of $j_{*}$ and $j_!,$ we have a
natural morphism
$$\varphi_Y:~~^p\mathcal{H}^0(j_{0!}\overline{\bbq}_l[m])\lra
{^p\mathcal{H}}^0(j_{0*}\overline{\bbq}_l[m]).$$ It induces an
intermediate extension functor
$$ j_{0!*}: \cp_U\lra\cp_Y$$
such that $j_{0!*}\overline{\bbq}_l[m]=\Im\{
^p\mathcal{H}^0(j_{0!}\overline{\bbq}_l[m])\lra
{^p\mathcal{H}}^0(j_{0*}\overline{\bbq}_l[m])\}. $ Furthermore,
$$p^*[d]\circ\varphi_Y: p^*[d](^p\mathcal{H}^0(j_{0!}\overline{\bbq}_l[m]))
\lra p^*[d](^p\mathcal{H}^0(j_{0*}\overline{\bbq}_l[m])).$$
 Since $p: X\lra Y$ is smooth of relative dimension
$d,$ $p^*[d]=p^{!}[-d]$ is $t$-exact , we have
$$\begin{array}{l}
\varphi_X: ~~^p\mathcal{H}^0(p^*j_{0!}\overline{\bbq}_l[m+d])\lra
{^p\mathcal{H}}^0(p^*j_{0*}\overline{\bbq}_l[m+d]),\\[3mm]
\varphi_X: ~~^p\mathcal{H}^0(j_{!}\overline{\bbq}_l[m+d])\lra
{^p\mathcal{H}}^0(j_{*}\overline{\bbq}_l[m+d]).
\end{array}$$

So $p^*[d](j_{0 !*}\overline{\bbq}_l[m])=j_{
!*}\overline{\bbq}_l[d+m].$ By the same argument, we get
$p^*[d]\circ F=F\circ p^*[d].$ Since $j_{0 !*}\overline{\bbq}_l$
has the purity property,  the statement of the lemma is true.
\end{proof}

 Let $\cn_{w_i}$ and $\cn_{w_i,3}$ be respectively the union
of orbits of regular modules of $\fk{C}(P,L)$ and $\fk{C}_0(P,L)$
with dimension vector $w_i\dz.$ Set $\cn_{{\bf
w}}=\cn_{w_1}\star\cdots\star\cn_{w_t}$ and $\cn_{{\bf
w},3}=\cn_{w_1,3}\star\cdots\star\cn_{w_t,3}.$
For any $P\in\cp_{prep}, M\in\bigoplus_{i=1}^l\ct_i,
I\in\cp_{prei},\pi_i\in\Pi_i^a,1\leqslant i\leqslant l,{\bf
w}\in{\bf P}(n), n\in\bbn,$ we define the varieties
\begin{eqnarray*}
\co_{P,\pi_1,\ldots,\pi_l,{\bf w},I}&=&\co_P\star\co_{\pi_1}\star\cdots\star\co_{\pi_l}
\star\cn_{{\bf w}}\star\co_I ,\\
\co_{P,M,{\bf
w},I}&=&\co_P\star\co_M\star\cn_{{\bf w},3}\star\co_I.\end{eqnarray*}

According to \cite{L4}, \cite{L5} and \cite{LXZ}, we know that
$IC({\overline{\co}_{P,\pi_1,\ldots,\pi_l,{\bf
w},I}},\overline{\bbq}_l)$ has the purity property. In order to
construct the canonical basis of $\ch^s(\llz)$
($\cong\mathbb{U}^+$), we need to study the purity property of
$\overline{\co}_{P,M,{\bf w},I}.$

\begin{theorem}\label{t:5.1.5} Let $Q$ be an affine quiver, and
let $X=\overline{\co}_{P,M,{\bf
w},I}$, for $P\in\cp_{prep}, M\in\bigoplus_{i=1}^l\ct_i,
I\in\cp_{prei}, {\bf w}\in {\bf P}(n), n\in\bbn.$ Then
$IC(X,\overline{\bbq}_l)$ are strong pure.
\end{theorem}

We postpone the proof of Theorem~\ref{t:5.1.5} to later sections.

\begin{corollary} Let $Q$ be an affine quiver, and
let $X=\overline{\co}_{P,M,{\bf
w},I}$, for $P\in\cp_{prep}, M\in\bigoplus_{i=1}^l\ct_i\bigoplus\bigoplus_{j=1}^t\mathscr{H}_{x_j},
I\in\cp_{prei}, {\bf w}\in {\bf P}(n), n\in\bbn,$ where $x_j$ are $\bbf_q$ rational points in $\mathbb{P}^1.$ Then
$IC(X,\overline{\bbq}_l)$ are strong pure.
\end{corollary}

\subsection{}\label{ss:canonical-basis}
Let
\begin{eqnarray}
b_{\overline{\co}_{P,M,{\bf w},I}}&=&
\sum_{i,N\in{\overline{\co}_{P,M,{\bf
w},I}}^F}v^{i+\dim\co_N-\dim\co_{P,M,{\bf
w},I}}\dim
\ch^i_{N}(IC({\overline{\co}_{P,M,{\bf
w},I}},\overline{\bbq}_l))\lr{N}.\label{f:7.2.1}
\end{eqnarray} Set
\begin{eqnarray*}{\bf CB}&=&\{b_{\overline{\co}_{P,M,{\bf w},I}}|P\in\mathcal {P}_{prep},
M\in\oplus_{i=1}^l\ct_i,I\in\mathcal {P}_{prei},{\bf w}\in {\bf
P}(n),n\in\bbn\}.\end{eqnarray*}
Then we have the main theorem of this paper

\begin{theorem}\label{t:7.2.1} Let $Q$ be an affine quiver,
 the set ${\bf CB}$ is the
canonical basis of $\ch^r(\llz)$ ($\cong\mathbb{U}^+$).
\end{theorem}


\begin{proof}
By Proposition~\ref{p:1.3.1}, Theorem~\ref{t:4.1.1} and
Theorem~\ref{t:5.1.5}, the proof is complete.
\end{proof}

\begin{corollary}
Let $Q$ be an affine quiver, there then exists the canonical basis of $\ch^{r,x_1,\cdots, x_t}(\llz).$
\end{corollary}

\section{Purity Properties of Perverse sheaves of closure
of semi-simple objects in $\ct_i$}

We start from this section the proof of Theorem~\ref{t:5.1.5}.
In this section we deal with the special case $\co_M$ for $M$ a semi-simple object in any non-homogeneous tube. We proceed type by and type.






\subsection{Type $\widetilde{A}_{n,1}$}

Let $Q$ be the quiver
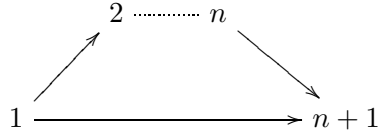
\[\xymatrix{&2\ar@{.}[r]&n\ar[dr]\\
1\ar[ur]\ar[rrr]&&&n+1}\]
There is only one non-homogeneous tube, which is of period $n$, and the regular simples $E_1,\ldots,E_n$ respectively have dimension vectors
$(1,0,0,\ldots,0,1),(0,1,0,\ldots,0,0),\ldots,(0,0,0,\ldots,1,0).$




\begin{lemma}\label{l:5.2.1}Let $X=\overline{\co}_{\bigoplus_{i=1}^nm_iE_i}, m_i\in\bbn.$ Then
$IC(X,\overline{\bbq}_l)$ is strong pure.
\end{lemma}


\begin{proof}
Let $V$ be the $I$-graded vector space
$V=\bigoplus_{i=1}^{n+1}V_i.$
Set
$\az=(m_1,m_2,\ldots,m_n,m_1)\in\bbn [I],$ then $\underline{\dim}V=\az,$ and
$$\begin{array}{l}
\bbe_{\az}=\bbe_{V}=\{x=(x_{12},x_{23},\ldots,x_{n,n+1},x_{1,n+1})|x_{i,i+1}\in M_{m_{i+1},m_i}(k), 1\leq i\leq n-1, \\
{\quad\quad\quad\quad\quad\quad\quad\quad\quad\quad\quad\quad\quad\quad\quad\quad\quad\quad\quad\quad} x_{n,n+1}\in M_{m_1,m_n}(k),x_{1,n+1}\in M_{m_1,m_1}(k)\},\\[5pt]
\co_{\bigoplus_{i=1}^nE_i}=\{x\in\bbe_{\az}|~x_{12}=x_{23}=\ldots=x_{n,n+1}=0,x_{1,n+1}\in GL_{m_1}(k)
\},\\[5pt]
X=\overline{\co}_{\bigoplus_{i=1}^nE_i}=\{x\in\bbe_{\az}|~x_{12}=x_{23}=\ldots=x_{n,n+1}=0\}=\bba^{m_1^2}.
\end{array}$$
 It is clear that
$IC_{X}(\overline{\bbq}_l)$ is strong pure. The statement is proved.
\end{proof}

\subsection{Type $\widetilde{D}_n$}

Let $Q$ be the quiver
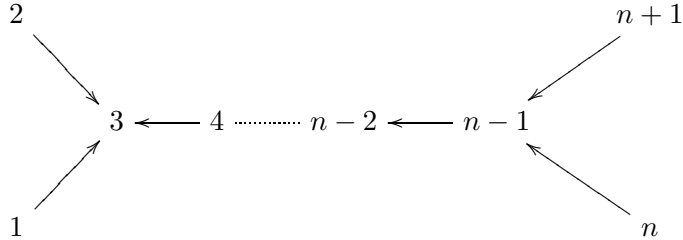
\[\xymatrix{ 2\ar[dr] &&&&& n+1\ar[dl]\\
&3&4\ar[l]\ar@{.}[r]& n-2 &n-1\ar[l]  &\\
 1\ar[ur] &&&&&n\ar[ul]}\]
There are three non-homogeneous tubes $\ct_1,\ct_2,\ct_3$, respectively of periods $2,2,n-2$. Let $E_1,E_2$ be the regular simples in $\ct_1$, let $E'_1,E'_2$ be the regular simples in $\ct_2$ and let $E''_1,\ldots,E''_{n-2}$ be the regular simples in $\ct_3$. Their dimension vectors are given as follows
\begin{itemize}
\item[$\ct_1$:] $(1,0,1,1\ldots,1,1,0)$, $(0,1,1,1,\ldots,1,0,1)$;
\item[$\ct_2$:] $(1,0,1,1\ldots,1,0,1)$, $(0,1,1,1,\ldots,1,1,0)$;
\item[$\ct_3$:] $(1,1,1,0\ldots,0,0,0)$, $(0,0,1,1,\ldots,1,1,1)$, $(0,0,0,1,0,\ldots,,0,0,0)$, $\ldots$, $(0,0,0,0,\ldots,0,1,0,0)$.
\end{itemize}


\begin{lemma}\label{l:5.2.2}
 Set
$X_1=\overline{\co}_{m_1E_1\oplus m_2E_2}, X_2=\overline{\co}_{l_1E'_1\oplus
l_2E'_2},$ and $X_3=\overline{\co}_{\bigoplus_{i=1}^{n-2}s_i E_i''}, \forall m_i,l_i,s_i\in\bbn.$
Then $IC(X_i,\overline{\bbq}_l)$ strong pure for any $i=1,2,3.$
\end{lemma}

\begin{proof}We only prove for $i=1.$ The other cases can be proved
similarly. Let $V$ be the $I$-graded vector space
$V=\bigoplus_{i=1}^{n+1}V_i,$ with $V_1=V_{n}=k^{m_1}, V_2=V_{n+1}=k^{m_2}$ and $V_i=k^{m_1+m_2}$
for all $3\leqslant i\leqslant n-1.$ Set
$\az=(m_1,m_2,m_1+m_2,\ldots,m_1+m_2,m_1,m_2)\in\bbn [I]$. Then
\begin{eqnarray*}
\bbe_{\az}=\bbe_{V}&\hspace{-8pt}=\hspace{-8pt}&\{x=(x_{13},x_{23},x_{43},\ldots,x_{n-1,n-2},x_{n,n-1},x_{n+1,n-1})|x_{13},x_{n,n-1}\in M_{m_1+m_2,m_1}(k),\\
                    &&{\quad\quad\quad\quad} x_{23},x_{n+1,n-1}\in M_{m_1+m_2,m_2}(k),x_{i+1,i}\in M_{m_1+m_2,m_1+m_2}(k), 3\leq i\leq n-2  \}.
\end{eqnarray*}
Thanks to \cite{[DR]}, we have $E_1\oplus E_2=M(y)$ for
$y\in\bbe_{\az}$ given by
\[y_{13}=y_{n,n-1}=\left(\begin{array}{l}
I_{m_1}\\
0\end{array}\right),~~y_{23}=y_{n+1,n-1}=\left(\begin{array}{l}
0\\
I_{m_2}\end{array}\right),~~y_{43}=\ldots=y_{n-1,n-2}=\left(\begin{array}{ll}
I_{m_1}&0\\
0&I_{m_2}\end{array}\right).\]


We claim that
\begin{eqnarray*}
\co_{E_1\oplus E_2}&=&\{x\in\bbe_{\az}|(x_{13}~~~x_{23})=x_{43}\cdots x_{n-1,n-2}\cdot(x_{n,n-1}~~~x_{n+1,n-1})\cdot
\left(\begin{array}{ll}
\xi_1&0\\
0&\xi_2\end{array}\right),\\
&&\xi_1\in GL_{m_1}(k),\xi_2\in GL_{m_2}(k),\det(x_{i,i-1})\neq 0, 4\leqslant i\leqslant
n-1,\det(x_{n,n-1}~~~x_{n+1,n-1})\neq 0\}
\\
&=&GL_{m_1}(k)\times
GL_{m_2}(k)\times\{x=(x_{43},\cdots,x_{n-1,n-2},x_{n,n-1},x_{n+1,n-1})|\\
&&\det(x_{i,i-1})\neq 0, 4\leqslant i\leqslant
n-1\,
\det(x_{n,n-1}~~~x_{n+1,n-1})\neq 0\}.
\end{eqnarray*}
Let $S$ be the set on the right hand side of the equality. For any $x\in\co_{E_1\oplus E_2}$ there exists $g=(g_i)_{i\in I}\in
GL_{\az}$ such that $x=g\bullet y$, i.e.
$$x_{13}=g_3\left(\begin{array}{l}
I_{m_1}\\
0\end{array}\right) g_1^{-1},~~x_{23}=g_3\left(\begin{array}{l}
0\\
I_{m_2}\end{array}\right) g_2^{-1},~~x_{i,i-1}=g_{i-1}g_i^{-1},4\leqslant
i\leqslant n-1,$$ $$x_{n,n-1}=g_{n-1}\left(\begin{array}{l}
I_{m_1}\\
0\end{array}\right) g_n^{-1},~~
  x_{n+1,n-1}=g_{n-1}\left(\begin{array}{l}
0\\
I_{m_2}\end{array}\right) g_{n+1}^{-1}.$$
The inclusion $\co_{E_1\oplus E_2}\subseteq S$ follows immediately, with $\xi_1=g_n g_1^{-1}$ and $\xi_2=g_{n+1}g_2^{-1}$.
%
Conversely, for an element in $S$, we have $x=g\bullet y$ for $g=(g_i)_{i\in I}\in GL_{\alpha}$ with $g_1=\xi_1^{-1}$, $g_2=\xi_2^{-1}$, $g_i=x_{i+1,i}\cdots x_{n-1,n-2} (x_{n,n-1}~~~x_{n+1,n-1})$ for $3\leq i\leq n-2$, $g_{n-1}=(x_{n,n-1}~~~x_{n+1,n-1})$ and $g_n=I_{m_1}$, $g_{n+1}=I_{m_2}$.

Let
\begin{eqnarray*}
X'=\{(x_{43},\ldots,x_{n-1,n-2},x_{n,n-1},x_{n+1,n-1})| x_{n,n-1}\in M_{m_1+m_2,m_1}(k),x_{n+1,n-1}\in M_{m_1+m_2,m_2}(k),\\
x_{i+1,i}\in M_{m_1+m_2,m_1+m_2}(k), 3\leq i\leq n-2
\}.
\end{eqnarray*}
Then
$$
X=\overline{\co}_{E_1\oplus E_2} =M_{m_1}(k)\times M_{m_2}(k)\times
X'=\bba^{(n-3)(m_1+m_2)^2+m_1^2+m_2^2}.
$$
The result is clear now.
\end{proof}

\subsection{Type $\widetilde{E}_{6}$}
Let $Q$ be the quiver
\[\xymatrix{&&7\ar[d]&&\\ &&6\ar[d]&&\\ 1\ar[r]&2\ar[r]&3&4\ar[l]&5\ar[l]}\]
There are three non-homogeneous tubes $\ct_1,\ct_2,\ct_3$, respectively of periods $2,3,3$. Let $E_1,E_2$ be the regular simples in $\ct_1$, let $E'_1,E'_2,E'_3$ be the regualr simples in $\ct_2$, and let $E''_1,E''_2,E''_3$ be the regular simples in $\ct_3$. Their dimension vectors are given as follows
\begin{itemize}
\item[$\ct_1$:] $(1,1,2,1,1,1,1)$, $(0,1,1,1,0,1,0)$;
\item[$\ct_2$:] $(1,1,1,1,0,0,0)$, $(0,1,1,0,0,1,1)$, $(0,0,1,1,1,1,0)$;
\item[$\ct_3$:] $(1,1,1,0,0,1,0)$, $(0,1,1,1,1,0,0)$, $(0,0,1,1,0,1,1)$.
\end{itemize}

\begin{lemma}\label{l:5.2.3}
Set $X_1=\overline{\co}_{E_1\oplus E_2},
X_2=\overline{\co}_{E'_1\oplus E'_2\oplus E'_3},$ and
$X_3=\overline{\co}_{E''_1\oplus E''_2\oplus E''_3}.$
 Then $IC(X_i,\overline{\bbq}_l)$ is strong pure for any $i=1,2,3$.
\end{lemma}

\begin{proof} Let $P(E_1), P(E_1')$ and $P(E_1'')$) be the indecomposable preprojective modules with dimension vectors $(0121111),(0111000)$ and $(0110010))$. Thanks to \cite{[DR]}, we have
$$ E_1=S_1\diamond P(E_1), E_1'=S_1\diamond P(E_1'), \text{~and~}E_1'' =S_1\diamond P(E_1''). $$
Moreover, $\co_{S_1\oplus P(E_1)}\subsetneqq\overline{\co}_{E_1}, \co_{S_1\oplus P(E_1')}\subsetneqq\overline{\co}_{E_1'}, \text{~and~} \co_{S_1\oplus P(E_1'')}\subsetneqq\overline{\co}_{E_1''}.$

We now only prove for $i=1.$ The other cases can be proved similarly. Let $V$ be the $I$-graded vector space
$V=\bigoplus_{i=1}^{7}V_i$ with $V_1=V_5=V_7=k,
V_2=V_4=V_6=k^2$ and $V_3=k^3$. Set $\gamma=(1,2,3,2,1,2,1)\in\bbn [I],$ then
\begin{eqnarray*}
\bbe_{\gamma}=\bbe_{V}&\hspace{-8pt}=\hspace{-8pt}&\{x|~x=(x_{12},x_{23},x_{43},x_{54},x_{63},x_{76}),\\
                    && x_{12},x_{54},x_{76}\in M_{21}(k),
x_{23},x_{43},x_{63}\in  M_{32}(k) \}.
\end{eqnarray*}
Thanks to \cite{[DR]}, we have $E_1\oplus E_2=M(y)$ for
$y\in\bbe_{\az}$ given by
\[y_{12}=y_{54}=y_{76}=\left(\begin{array}{l}
0\\
1\end{array}\right),~~ y_{23}=\left(\begin{array}{ll}
1&0\\
0&1\\
0&0\end{array}\right),~~y_{43}=\left(\begin{array}{ll}
1&0\\
0&0\\
0&1\end{array}\right),~~y_{63}=\left(\begin{array}{ll}
1&0\\
0&1\\
0&1\end{array}\right).\]

Let $\alpha=\underline{\dim}E_1,$ and $\beta= \underline{\dim}E_2,$ then $\gamma=\alpha+\beta.$ From {\bf 1.3}, we have the diagram
$$\xymatrix{\bbe_{\az}\times\bbe_{\bz} & \bbe' \ar[l]_(0.35){p_1}\ar[r]^{p_2}&\bbe''\ar[r]^{p_3}& \bbe_{\gz}.}$$

Let
 $x_{12}=\left[\begin{array}{l}
\lambda\\
\mu
\end{array}\right],$
We may define a homomorphsim
$$\varphi: \bbe'\longrightarrow \bbe'$$
by $\varphi(x,W,R',R'')=(y,W,R',R''),$ where $y_{12}=\left[\begin{array}{l}
\lambda\\
0
\end{array}\right], x_{23}=y_{23}, x_{43}=y_{43}, x_{54}=y_{54}, x_{63}=y_{63}, \text{~and~}x_{76}=y_{76}.$  Then there are the homomorphisms $\varphi_{\alpha\times\beta}, \varphi'',$ and $\phi$ induced by $\varphi.$ We therefore have a commutative diagram
\[
 \xymatrix{\bbe_\alpha\times\bbe_\beta\ar[d]^{\varphi_{\alpha\times\beta}}\ar[r]^(0.65){p_1} & \bbe'\ar[d]^{\varphi}\ar[r]^(0.65){p_2}&\bbe''\ar[d]^{\varphi''}\ar[r]^(0.65){p_3}&\bbe_\gamma\ar[d]^{\phi}\\
\bbe_\alpha\times\bbe_\beta \ar[r]^{p_1} &\bbe'\ar[r]^{p_2}&\bbe''\ar[r]^{p_2}&\bbe_\gamma.}
\]

Let $X=\phi(\bbe\gamma),$ we then get a morphism from $\bbe\gamma$ to $X$ which is also denoted by $\phi$. It is clear that $\phi$ is smooth with relative dimension $1.$

Since $\phi(E_1\oplus E_2)=P(E_1)\oplus S_1\oplus E_2,$ and $\phi$ is a closed map, we get $\phi(\co_{(E_1\oplus E_2)})\supseteq\co_{(p(E_1)\oplus S_1\oplus E_2)}$ and $\phi(\overline{\co}_{(E_1\oplus E_2)})\supseteq\overline{\co}_{(P(E_1)\oplus S_1\oplus E_2)}.$

 Since $\overline{\co}_{(E_1\oplus E_2)}$ and $\phi(\overline{\co}_{(E_1\oplus E_2)})$ are irreducible subvarieties, by the theorem of upper semicontinuity of dimension, we have
$$\dim \overline{\co}_{(E_1\oplus E_2)}-\dim\phi(\overline{\co}_{(E_1\oplus E_2)})=1 ,\text{~and~}
\dim\phi(\overline{\co}_{(E_1\oplus E_2)})= \dim\overline{\co}_{(E_1\oplus E_2)}-1.$$

Since
$$Hom_{\llz}(P(E_1),S_1)=Hom_{\llz}(E_2,S_1)=0, $$
and
$$\langle \underline{\dim}P(E_1), \underline{\dim}E_2\rangle=0= \dim Hom_{\llz}(P(E_1),E_2),$$
we have

$\dim End_{\llz}(E_1\oplus E_2)=2,$ and $\dim End_{\llz}(P(E_1)\oplus S_1\oplus E_2)=3.$

It follows that
$$\dim \overline{\co}_{(E_1\oplus E_2)}=\dim\overline{\co}_{(P(E_1)\oplus S_1\oplus E_2)}+1,\text{~and~}
\dim \overline{\co}_{(P(E_1)\oplus S_1\oplus E_2)}=\dim\phi(\overline{\co}_{(E_1\oplus E_2)}).$$

Since $\phi(\overline{\co}_{(E_1\oplus E_2)})$ is irreducible, we get
$\phi(\overline{\co}_{(E_1\oplus E_2)})=\overline{\co}_{(P(E_1)\oplus S_1\oplus E_2)}.$

Since $P(E_1)\oplus S_1\oplus E_2$ is aperiodic, $IC(\overline{\co}_{P(E_1)\oplus S_1\oplus E_2},\overline{\bbq}_l)$ is strong pure, by Theorem 5.4 in \cite{L4}.
Therefore, by Lemma 6.1.4, the statement of lemma is true.

\end{proof}

\subsection{Type $\widetilde{E}_{7}$}
Let $Q$ be the quiver
\[\xymatrix{&&&8\ar[d]&&&\\ 1\ar[r]&2\ar[r]&3\ar[r]&4&5\ar[l]&6\ar[l]&7\ar[l]}\]
There are three non-homogeneous tubes $\ct_1,\ct_2,\ct_3$, respectively of periods $2,3,4$. Let $E_1,E_2$ be the regular simples in $\ct_1$, let $E'_1,E'_2,E'_3$ be the regualr simples in $\ct_2$, and let $E''_1,E''_2,E''_3.E''_4$ be the regular simples in $\ct_3$. Their dimension vectors are given as follows
\begin{itemize}
\item[$\ct_1$:] $(1,1,2,2,1,1,0,1)$, $(0,1,1,2,2,1,1,1)$;
\item[$\ct_2$:] $(1,1,1,2,1,1,1,1)$, $(0,1,1,1,1,1,0,0)$, $(0,0,1,1,1,0,0,1)$;
\item[$\ct_3$:] $(1,1,1,1,1,0,0,0)$, $(0,1,1,1,0,0,0,1)$, $(0,0,1,1,1,1,1,0)$, $(0,0,0,1,1,1,0,1)$.
\end{itemize}

\begin{lemma}\label{l:5.2.4}
Set $X_1=\overline{\co}_{E_1\oplus E_2},
X_2=\overline{\co}_{E'_1\oplus E'_2\oplus E'_3},$ and
$X_3=\overline{\co}_{E''_1\oplus E''_2\oplus E''_3\oplus E''_4}.$
 Then $IC(X_i,\overline{\bbq}_l)$ is strong pure for any $i=1,2,3.$
\end{lemma}

\begin{proof}  Let $P(E_1), P(E_1')$ and $P(E_1'')$) be the indecomposable preprojective modules with dimension vectors $(01221101), (01121111)$ and $(01111000)).$ Thanks to \cite{[DR]}, we have
$$ E_1=S_1\diamond P(E_1), E_1'=S_1\diamond P(E_1'), \text{~and~}E_1'' =S_1\diamond P(E_1''). $$
Moreover, $\co_{S_1\oplus P(E_1)}\subsetneqq\overline{\co}_{E_1}, \co_{S_1\oplus P(E_1')}\subsetneqq\overline{\co}_{E_1'}, \text{~and~} \co_{S_1\oplus P(E_1'')}\subsetneqq\overline{\co}_{E_1''}.$

We now only prove for $i=3.$ The other cases can be proved
similarly. Let $V$ be the $I-$graded vector space
$V=\oplus_{i=1}^{8}V_i,$ and $V_1=V_7=k,
V_2=V_6=V_8=k^2,V_3=V_5=k^3,V_4=k^4.$ Set
$\gamma=(1,2,3,4,3,2,1,2)\in\bbn [I],$ then
\begin{eqnarray*}
\bbe_{\gamma}=\bbe_{V}&=&\{x|~x=(x_{12},x_{23},x_{34},x_{54},x_{65},x_{76},x_{84}),\\
                    && x_{12},x_{76}\in M_{21}(k),
x_{23},x_{65}\in  M_{32}(k), x_{34},x_{54}\in  M_{43}(k),x_{84}\in
M_{42}(k)\}.
\end{eqnarray*}
Thanks to \cite{[DR]}, we can get $E''_1\oplus E''_2\oplus
E''_3\oplus E''_4=M(y),$ where $y\in\bbe_{\az}$ and

$$\begin{array}{l}
y_{12}=y_{76}=\left[\begin{array}{l}
0\\
1\end{array}\right],~y_{23}=\left[\begin{array}{ll}
1&0\\
0&0\\
0&1\end{array}\right],~ y_{65}=\left[\begin{array}{ll}
1&0\\
0&1\\
0&0\end{array}\right], ~y_{84}=\left[\begin{array}{ll}
1&0\\
0&0\\
0&1\\
0&0\end{array}\right],\\
 y_{34}=\left[\begin{array}{lll}
1&0&0\\
0&1&0\\
0&0&0\\
0&0&1\end{array}\right],~y_{54}=\left[\begin{array}{lll}
0&0&0\\
1&0&0\\
0&1&0\\
0&0&1\end{array}\right].
\end{array}$$

Let $\alpha=\underline{\dim}E_1'',$ and $\beta= \sum_{i=2}^4\underline{\dim}E_i'',$ then $\gamma=\alpha+\beta.$ From {\bf 1.3}, we have the diagram
$$\xymatrix{\bbe_{\az}\times\bbe_{\bz} & \bbe' \ar[l]_(0.35){p_1}\ar[r]^{p_2}&\bbe''\ar[r]^{p_3}& \bbe_{\gz}.}$$

Let
 $x_{12}=\left[\begin{array}{l}
\lambda\\
\mu
\end{array}\right],$
 as in lemma 7.3.1, we may also define a homomorphsim
$$\varphi: \bbe'\longrightarrow \bbe'$$
by $\varphi(x,W,R',R'')=(y,W,R',R''),$ where $$y_{12}=\left[\begin{array}{l}
\lambda\\
0
\end{array}\right], x_{23}=y_{23}, x_{34}=y_{34}, x_{54}=y_{54}, x_{65}=y_{65}, x_{76}=y_{76},\text{~and~} x_{84}=y_{84}.$$ Then there are the homomorphisms $\varphi_{\alpha\times\beta},\varphi'',$ and $\phi$ induced by $\varphi.$ We therefore have a commutative diagram
\[
 \xymatrix{\bbe_\alpha\times\bbe_\beta\ar[d]^{\varphi_{\alpha\times\beta}}\ar[r]^(0.65){p_1} & \bbe'\ar[d]^{\varphi}\ar[r]^(0.65){p_2}&\bbe''\ar[d]^{\varphi''}\ar[r]^(0.65){p_3}&\bbe_\gamma\ar[d]^{\phi}\\
\bbe_\alpha\times\bbe_\beta \ar[r]^{p_1} &\bbe'\ar[r]^{p_2}&\bbe''\ar[r]^{p_2}&\bbe_\gamma.}
\]

Let $X=\phi(\bbe\gamma)$, we then get a morphism from $\bbe\gamma$ to $X$ which is also denoted by $\phi$. It is clear
that $\phi$ is smooth with relative dimension $1.$

Since $\phi(E_1''\oplus\oplus_{i=2}^4 E_i'')=P(E_1'')\oplus S_1\oplus \oplus_{i=2}^4 E_i'',$ and $\phi$ is a closed map, we get $\phi(\co_{(E_1''\oplus\oplus_{i=2}^4 E_i'')})\supseteq\co_{(P(E_1'')\oplus S_1\oplus \oplus_{i=2}^4 E_i'')}$ and $\phi(\overline{\co}_{(E_1''\oplus\oplus_{i=2}^4 E_i'')})\supseteq\overline{\co}_{(P(E_1'')\oplus S_1\oplus \oplus_{i=2}^4 E_i'')}.$

Since $\overline{\co}_{(E_1''\oplus\oplus_{i=2}^4 E_i'')}$ and $\phi(\overline{\co}_{E_1''\oplus\oplus_{i=2}^4 E_i'')})$ are irreducible subvarieties, by the theorem of upper semicontinuity of dimension, we have
$$\dim \overline{\co}_{(E_1''\oplus\oplus_{i=2}^4 E_i'')}-\dim\phi(\overline{\co}_{( E_1''\oplus\oplus_{i=2}^4 E_i'')})=1 ,\text{~and~}
\dim\phi(\overline{\co}_{(E_1''\oplus\oplus_{i=2}^4 E_i'')})= \dim\overline{\co}_{(E_1''\oplus\oplus_{i=2}^4 E_i'')}-1.$$

Since
$$Hom_{\llz}(P(E_1),S_1)=Hom_{\llz}(E_i'',S_1)=0,  i=2,3,4,$$
and
$$\langle \underline{\dim}P(E_1), \underline{\dim}E_i''\rangle=0= \dim Hom_{\llz}(P(E_1),E_i''), i=2,3,4,$$
we have

$\dim End_{\llz}(E_1''\oplus\oplus_{i=2}^4 E_i'')=4,$ and $\dim End_{\llz}(P(E_1'')\oplus S_1\oplus \oplus_{i=2}^4 E_i'')=5.$

It follows that
$$\dim \overline{\co}_{(E_1''\oplus\oplus_{i=2}^4 E_i'')}=\dim\overline{\co}_{(P(E_1'')\oplus S_1\oplus \oplus_{i=2}^4 E_i'')}+1,\text{~and~}
\dim \overline{\co}_{(P(E_1'')\oplus S_1\oplus \oplus_{i=2}^4 E_i'')}=\dim\phi(\overline{\co}_{(E_1''\oplus\oplus_{i=2}^4 E_i'')}).$$

Since $\phi(\overline{\co}_{(E_1''\oplus\oplus_{i=2}^4 E_i'')})$ is irreducible, we have $\phi(\overline{\co}_{(E_1''\oplus\oplus_{i=2}^4 E_i'')})=\overline{\co}_{(P(E_1'')\oplus S_1\oplus \oplus_{i=2}^4 E_i'')}.$

Since $P(E_1'')\oplus S_1\oplus \oplus_{i=2}^4 E_i''$ is aperiodic, $IC(\overline{\co}_{P(E_1'')\oplus S_1\oplus \oplus_{i=2}^4 E_i''},\overline{\bbq}_l)$ is strong pure, by Theorem 5.4 in \cite{L4}.
Therefore, by Lemma 6.1.4, the statement of lemma is true.

\end{proof}

\subsection{Type $\widetilde{E}_{8}$}
Let $Q$ be the quiver
\[ \xymatrix{&&9\ar[d]&&&&&\\ 1\ar[r]&2\ar[r]&3&4\ar[l]& 5\ar[l]& 6\ar[l]& 7\ar[l]&8\ar[l]}\]
There are three non-homogeneous tubes $\ct_1,\ct_2,\ct_3$, respectively of periods $2,3,5$. Let $E_1,E_2$ be the regular simples in $\ct_1$, let $E'_1,E'_2,E'_3$ be the regualr simples in $\ct_2$, and let $E''_1,E''_2,E''_3.E''_4,E''_5$ be the regular simples in $\ct_3$. Their dimension vectors are given as follows
\begin{itemize}
\item[$\ct_1$:] $(1,2,3,2,2,1,1,0,2)$, $(1,2,3,3,2,2,1,1,1)$;
\item[$\ct_2$:] $(1,2,2,1,1,1,0,0,1)$, $(0,1,2,2,2,1,1,1,1)$, $(1,1,2,2,1,1,1,0,1)$;
\item[$\ct_3$:] $(1,1,1,1,1,0,0,0,0)$, $(0,1,1,1,0,0,0,0,1)$, $(1,1,2,1,1,1,1,1,1)$,$(0,1,1,1,1,1,1,0,0)$,\\ $(0,0,1,1,1,1,0,0,1)$.
\end{itemize}

\begin{lemma}\label{l:5.2.5}
Set $X_1=\overline{\co}_{E_1\oplus E_2},
X_2=\overline{\co}_{E'_1\oplus E'_2\oplus E'_3},$ and
$X_3=\overline{\co}_{E''_1\oplus E''_2\oplus E''_3\oplus E''_4\oplus
E''_5}.$
Then $IC(X_i,\overline{\bbq}_l)$ has the purity property for any $i=1,2,3.$
\end{lemma}

\begin{proof} Let $P(E_2), P(E_2')$ and $P(E_3'')$) be the indecomposable preprojective modules with dimension vectors $(123322101), (012221101)$ and $(112111101)).$ Thanks to \cite{[DR]}, we have
$$ E_2=S_8\diamond P(E_2), E_2'=S_8\diamond P(E_2'), \text{~and~}E_3'' =S_8\diamond P(E_3''). $$
Moreover, $\co_{S_8\oplus P(E_2)}\subsetneqq\overline{\co}_{E_2}, \co_{S_8\oplus P(E_2')}\subsetneqq\overline{\co}_{E_2'}, \text{~and~} \co_{S_8\oplus P(E_3'')}\subsetneqq\overline{\co}_{E_3''}.$

We now only prove for $i=1.$ The other cases can be proved
similarly. Let $V$ be the $I-$graded vector space
$V=\oplus_{i=1}^{9}V_i,$ and $V_1=V_7=k^2,
V_2=V_5=k^4,V_3=k^6,V_4=k^5,V_6=V_9=k^3,V_8=k.$ Set
$\gamma=(2,4,6,5,4,3,2,1,3)\in\bbn [I],$ then
\begin{eqnarray*}
\bbe_{\gamma}=\bbe_{V}&=&\{x|~x=(x_{12},x_{23},x_{43},x_{54},x_{65},x_{76},x_{87},x_{93}),\\
                    && x_{12}\in M_{42}(k),x_{23}\in
M_{64}(k),x_{43}\in M_{65}(k) ,x_{54}\in M_{54}(k) ,x_{65}\in
M_{43}(k) ,x_{76}\in M_{32}(k),\\
&&x_{87}\in M_{21}(k),x_{93}\in M_{63}(k)\}.
\end{eqnarray*}
Thanks to \cite{[DR]}, we can get $E_1\oplus E_2=M(y),$ where
$y\in\bbe_{\gamma}$ and

$$\begin{array}{l}
y_{12}=\left[\begin{array}{ll} 0&0\\1&0\\0&0\\0&1
\end{array}\right],~y_{23}=\left[\begin{array}{llll}
0&0&0&0\\
1&0&0&0\\
0&1&0&0\\
0&0&0&0\\
0&0&1&0\\
0&0&0&1\end{array}\right],~ y_{43}=\left[\begin{array}{lllll}
1&0&0&0&0\\
0&1&0&0&0\\
0&0&0&0&0\\
0&0&1&0&0\\
0&0&0&1&0\\
0&0&0&0&1\end{array}\right], ~y_{54}=\left[\begin{array}{llll}
1&0&0&0\\
0&1&0&0\\
0&0&1&0\\
0&0&0&1\\
0&0&0&0\end{array}\right],\\
 y_{65}=\left[\begin{array}{lll}
1&0&0\\
0&0&0\\
0&1&0\\
0&0&1\end{array}\right],~y_{76}=\left[\begin{array}{ll}
1&0\\
0&1\\
0&0\end{array}\right],~y_{87}=\left[\begin{array}{l}
0\\
1\end{array}\right],~y_{93}=\left[\begin{array}{lll}
1&0&0\\
1&1&0\\
0&1&0\\
0&0&1\\
0&0&1\\
0&0&1.\end{array}\right].
\end{array}$$

Let $\alpha=\underline{\dim}E_2,$ and $\beta= \underline{\dim}E_1,$ then $\gamma=\alpha+\beta.$ From {\bf 1.3}, we have the diagram
$$\xymatrix{\bbe_{\az}\times\bbe_{\bz} & \bbe' \ar[l]_(0.35){p_1}\ar[r]^{p_2}&\bbe''\ar[r]^{p_3}& \bbe_{\gz}.}$$

Let
 $x_{87}=\left[\begin{array}{l}
\lambda\\
\mu
\end{array}\right],$
we define a homomorphsim
$$\varphi: \bbe'\longrightarrow \bbe'$$
by $\varphi(x,W,R',R'')=(y,W,R',R''),$ where $y_{87}=\left[\begin{array}{l}
\lambda\\
0
\end{array}\right], x_{12}=y_{12}, x_{23}=y_{23}, x_{43}=y_{43}, x_{54}=y_{54}, x_{65}=y_{65},x_{76}=y_{76}, \text{~and~}x_{93}=y_{93}.$  Then there are the homomorphisms $\varphi_{\alpha\times\beta}, \varphi'',$ and $\phi$ induced by $\varphi.$ We therefore have a commutative diagram
\[
 \xymatrix{\bbe_\alpha\times\bbe_\beta\ar[d]^{\varphi_{\alpha\times\beta}}\ar[r]^(0.65){p_1} & \bbe'\ar[d]^{\varphi}\ar[r]^(0.65){p_2}&\bbe''\ar[d]^{\varphi''}\ar[r]^(0.65){p_3}&\bbe_\gamma\ar[d]^{\phi}\\
\bbe_\alpha\times\bbe_\beta \ar[r]^{p_1} &\bbe'\ar[r]^{p_2}&\bbe''\ar[r]^{p_2}&\bbe_\gamma.}
\]

Let $X=\phi(\bbe\gamma)$, we then get a morphism from $\bbe\gamma$ to $X$ which is also denoted by $\phi$. It is clear
that $\phi$ is smooth with relative dimension $1.$

Since $\phi(E_2\oplus E_1)=P(E_2)\oplus S_8\oplus E_1,$ and $\phi$ is a closed map, we get $\phi(\co_{(E_1\oplus E_2)})\supseteq\co_{(p(E_2)\oplus S_8\oplus E_1)}$ and $\phi(\overline{\co}_{(E_1\oplus E_2)})\supseteq\overline{\co}_{(P(E_2)\oplus S_8\oplus E_1)}.$

Since $\overline{\co}_{(E_1\oplus E_2)}$ and $\phi(\overline{\co}_{(E_1\oplus E_2)})$ are irreducible subvarieties, by the theorem of upper semicontinuity of dimension, we have
$$\dim \overline{\co}_{(E_1\oplus E_2)}-\dim\phi(\overline{\co}_{(E_1\oplus E_2)})=1 ,\text{~and~}
\dim\phi(\overline{\co}_{(E_1\oplus E_2)})= \dim\overline{\co}_{(E_1\oplus E_2)}-1.$$

Since
$$Hom_{\llz}(P(E_2),S_8)=Hom_{\llz}(E_1,S_8)=0, $$
and
$$\langle \underline{\dim}P(E_2), \underline{\dim}E_1\rangle=0= \dim Hom_{\llz}(P(E_2),E_1),$$
we have

$\dim End_{\llz}(E_1\oplus E_2)=2,$ and $\dim End_{\llz}(P(E_2)\oplus S_8\oplus E_1)=3.$

It follows that
$$\dim \overline{\co}_{(E_1\oplus E_2)}=\dim\overline{\co}_{(P(E_2)\oplus S_8\oplus E_1)}+1,\text{~and~}
\dim \overline{\co}_{(P(E_2)\oplus S_8\oplus E_1)}=\dim\phi(\overline{\co}_{(E_1\oplus E_2)}).$$

Since $\phi(\overline{\co}_{(E_1\oplus E_2)})$ is irreducible, we get $\phi(\overline{\co}_{(E_1\oplus E_2)})=\overline{\co}_{(P(E_2)\oplus S_8\oplus E_1)}.$

Since $P(E_2)\oplus S_8\oplus E_1$ is aperiodic, $IC(\overline{\co}_{P(E_2)\oplus S_8\oplus E_1},\overline{\bbq}_l)$ is strong pure, by Theorem 5.4 in \cite{L4}.
Therefore, by Lemma 6.1.4, the proof is complete.
\end{proof}

Furthermore,  we can also prove that the closure of
orbits of semi-simple objects are strong pure for the type $E_6, E_7,$ and $E_8$.

\section{$IC({\overline{\co}_{P,M,{\bf
w},I}},\overline{\bbq}_l))$ are strong pure}

In this section, we will show that $IC({\overline{\co}_{P,M,{\bf
w},I}},\overline{\bbq}_l))$ is strong pure, that is, the aim of this subsection is to prove Theorem~\ref{t:5.1.5}.

\subsection{}
 From \cite{L4},\cite{L5} and \cite{LXZ}, we know that
$IC({\overline{\co}_{P,\pi_1,\ldots,\pi_l,{\bf
w},I}},\overline{\bbq}_l)$ is strong pure. For $M\in \bbe_\alpha,$ let $\co_M\subset \bbe_\alpha$ be the $G_\alpha-$orbit of $M.$ We take $\mathbf{1}_M\in\bbc_G(V_\alpha)$ to be the characteristic function of $\co_M,$ and set $f_M=v_q^{-dim\co_M}\mathbf{1}_M.$ We
consider the algebra $(\mathbf{L},*)$ generated by $f_M$ over $\bbq[v_q,{v_q}^{-1}],$ for all $M\in\bbe_\alpha$ and all $\alpha\in{\bbn} I.$

\begin{proposition}\cite{LXZ} {\sl The linear map $\chi: \ch^*(\llz)\rightarrow (\mathbf{L},*)$  defined by
$$\chi(\langle M\rangle)=f_M, \text{~ for all~} M\in\cp.$$
is an isomorphism of associative $\bbq[v_q,{v_q}^{-1}]-$algebras.~}\end{proposition}

\begin{lemma} {\sl Let $M \in \ct_i$ and let $Z=\overline{\co}_M$. Then $IC(Z,\overline{\bbq}_l)$ is strong pure.~}
\end{lemma}

\begin{proof} According to Proposition 3.4 in \cite{GJ}, there is an regular semi-simple object $M_1$ in $\ct_i$ such that $g^M_{M_1 M_2}=1,$ and $M_1\diamond M_2=M.$

 By Lemma 2.3(vi) in \cite{LXZ} and \cite{GJ}, we have
 \begin{eqnarray}<M_1>\ast <M_2> &=& <M>+\sum_{Y,\overline{\co}_Y\subsetneqq \overline{\co}_M} a_Y<Y>,\end{eqnarray}
 where $a_Y=v^{-\dim\co_{M_1}-\dim\co_{M_2}+\dim\co_{Y}-{\bf m}(\underline{\dim}M_1,\underline{\dim}M_2)}\varphi^Y_{M_1M_2}(v^2)\in \bbz [v,v^{-1}]$

 Based on section 7, we known that $M_1$ is strong pure, and by induction hypothesis, we get that $M_2$ is strong pure. Let $q=p^r, r\in\bbn, $ then we have

\begin{eqnarray}
b_{\overline{\co}_{M_1}}&=&
\langle M_1\rangle+\sum_{i,\overline{\co}_N\subsetneqq\overline{\co}_{M_1}}v^{i+\dim\co_N-\dim\co_{M_1}}\dim
\ch^i_{N}(IC({\overline{\co}_{M_1}},\overline{\bbq}_l))\lr{N},\\
b_{\overline{\co}_{M_2}}&=& \langle M_2\rangle+
\sum_{i,\overline{\co}_N\subsetneqq \overline{\co}_{M_2}}v^{i+\dim\co_N-\dim\co_{M_2}}\dim
\ch^i_{N}(IC({\overline{\co}_{M_2}},\overline{\bbq}_l))\lr{N},\text{~and}\\
\chi(b_{\overline{\co}_{M}})&=& \sum_{\overline{\co}_N\subsetneqq\overline{\co}_{M}}(-\sqrt{p}^r)^{-\dim\co_M+\dim\co_N}Tr((Fr^*)^r,\ch^*_N(\overline{\co}_{M}))f_N.
\end{eqnarray}

According to \cite{L4}, (14), and Proposition 6.1.1 , we have
\begin{eqnarray}
 b_{\overline{\co}_{M_1}}*b_{\overline{\co}_{M_2}}&\stackrel{{[L4],(14)}}=&b_{\overline{\co}_{M}}+\sum_{\overline{\co}_X\subsetneqq\overline{\co}_M}(\sum_{j\in\bbz}\dim D^j_{P_{M_1},P_{M_2},P_{X}}v^j)b_{\overline{\co}_{X}}\\
b_{\overline{\co}_{M}}&=&b_{\overline{\co}_{M_1}}*b_{\overline{\co}_{M_2}}-\sum_{\overline{\co}_X\subsetneqq\overline{\co}_ M}(\sum_{j\in\bbz}\dim D^j_{P_{M_1},P_{M_2},P_{X}}v^j)b_{\overline{\co}_{X}}.
\end{eqnarray}

We take 

$A_1=\sum_{i,\overline{\co}_{N_1}\subsetneqq\overline{\co}_{M_1}}v^{i+\dim\co_{N_1}-\dim\co_{M_1}}\dim
\ch^i_{N_1}(IC({\overline{\co}_{M_1}},\overline{\bbq}_l))\lr{N_1}, \\
{\quad}A_2=\sum_{j,\overline{\co}_{N_2}\subsetneqq \overline{\co}_{M_2}}v^{j+\dim\co_{N_2}-\dim\co_{M_2}}\dim
\ch^j_{{N_2}}(IC({\overline{\co}_{M_2}},\overline{\bbq}_l))\lr{{N_2}},\\
{\quad}\theta_Y=v^{-\dim\co_{M_1}-\dim\co_{M_2}+\dim\co_{Y}-{\bf m}(\underline{\dim}M_1,\underline{\dim}M_2)}\varphi^Y_{M_1M_2}(v^2),\\
{\quad}\lambda_Y=v^{-\dim\co_{M_1}-\dim\co_{M_2}+\dim\co_{Y}-{\bf m}(\underline{\dim}M_1,\underline{\dim}M_2)}(\sum_{j,\overline{\co}_{N_2}\subsetneqq \overline{\co}_{M_2}}v^j\dim\ch^j_{N_2}(\overline{\co}_{M_2})\varphi^Y_{M_1N_2}(v^2)),\\
{\quad}\mu_Y=v^{-\dim\co_{M_1}-\dim\co_{M_2}+\dim\co_{Y}-{\bf m}(\underline{\dim}M_1,\underline{\dim}M_2)}(\sum_{i,\overline{\co}_{N_1}\subsetneqq \overline{\co}_{M_1}}v^i\dim\ch^i_{N_1}(\overline{\co}_{M_1})\varphi^Y_{N_1M_2}(v^2)),\\
{\quad}\nu_Y=v^{-\dim\co_{M_1}-\dim\co_{M_2}+\dim\co_{Y}-{\bf m}(\underline{\dim}M_1,\underline{\dim}M_2)}\\
{\quad\quad\quad\quad\quad\quad\quad\quad}(\sum_{i,j,\overline{\co}_{N_1}\subsetneqq \overline{\co}_{M_1}, \overline{\co}_{N_2}\subsetneqq \overline{\co}_{M_2}}v^{i+j}\dim\ch^i_{N_1}(\overline{\co}_{M_1})\dim\ch^j_{N_2}(\overline{\co}_{
M_2})
\varphi^Y_{N_1N_2}(v^2)).$

By Theorem 5.1.1, we have 

$ b_{\overline{\co}_{M_1}}*b_{\overline{\co}_{M_2}}=(\langle M_1\rangle+A_1)
*(\langle M_2\rangle+
A_2)=\langle M_1\rangle*\langle M_2\rangle+\langle M_1\rangle*A_2+A_1*\langle M_2\rangle+A_1*A_2.
$

Thus
\begin{eqnarray} b_{\overline{\co}_{M_1}}*b_{\overline{\co}_{M_2}}&=&\langle M\rangle+\sum_{Y,\overline{\co}_Y \subsetneqq\overline{\co}_M}\sum_Y(\theta_Y+\lambda_Y+\mu_Y+\nu_Y)\langle Y\rangle
\end{eqnarray}

where $c_Y=\theta_Y+\lambda_Y+\mu_Y+\nu_Y\in\bbq[v,v^{-1}].$

By induction hypothesis, $IC(\overline{\co}_{X},\overline{\bbq}_l)$ are strong pure. We get that
\begin{eqnarray}
\chi(b_{\overline{\co}_{M_1}})&=&
f_{M_1}+\sum_{i,\overline{\co}_N\subsetneqq \overline{\co}_{M_1}}(-\sqrt{p}^r)^{i+\dim\co_N-\dim\co_{M_1}}\dim
\ch^i_{N}(IC({\overline{\co}_{M_1}},\overline{\bbq}_l))f_{N},\\
\chi(b_{\overline{\co}_{M_2}})&=& f_{M_2}+
\sum_{i,\overline{\co}_N\subsetneqq \overline{\co}_{M_2}}(-\sqrt{p}^r)^{i+\dim\co_N-\dim\co_{M_2}}\dim
\ch^i_{N}(IC({\overline{\co}_{M_2}},\overline{\bbq}_l))f_N,\\
\chi(b_{\overline{\co}_{X}})&=& f_{X}+
\sum_{i,\overline{\co}_Y\subsetneqq\overline{\co}_ X}(-\sqrt{p}^r)^{i+\dim\co_Y-\dim\co_{X}}\dim
\ch^i_{Y}(IC({\overline{\co}_{X}},\overline{\bbq}_l))f_Y.
\end{eqnarray}

By Proposition 8.1.1, we have
$$\chi(b_{\overline{\co}_{M}})=\chi(b_{\overline{\co}_{M_1}})*\chi(b_{\overline{\co}_{M_2}})-\sum_{\overline{\co}_X\subsetneqq\overline{\co}_ M}(\sum_{j\in\bbz}\dim D^j_{P_{M_1},P_{M_2},P_{X}}v^j))\chi(b_{\overline{\co}_{X}}),$$

From (17), (20), (21), (22), (23), and Theorem 5.1.1, we get

$\sum_{\overline{\co}_N\subsetneqq \overline{\co}_{M}}((-\sqrt{p})^r)^{-\dim\co_M+\dim\co_N}Tr((Fr^*)^r,\ch^*_N(\overline{\co}_{M}))f_N
\\{\quad\quad\quad\quad\quad\quad\quad}=\sum_{\overline{\co}_N\subsetneqq \overline{\co}_{M}}a_N((\sqrt{p})^r,(\sqrt{p})^{-r})f_N, 
$

where $a_N\in\bbq[\sqrt{p},\sqrt{p}^{-1}], a_X(v,v^{-1})$ is Laurent polynomial.

Moreover,
\begin{eqnarray}
((-\sqrt{p})^r)^{-\dim\co_M+\dim\co_N}Tr((Fr^*)^r,\ch^*_N(\overline{\co}_{M}))
= a_N((\sqrt{p})^r,(\sqrt{p})^{-r}),
\end{eqnarray}
where $a_N\in\bbq[\sqrt{p},\sqrt{p}^{-1}],$ for all $r\geq 1.$

According to Lemma 6.1.2 and 6.1.3, we known that $IC(\overline{\co}_{M},\overline{\bbq}_l)$ is very pure. If $\lambda'$ (resp. $\lambda''$) is an eigenvalue of $Fr^{*}$ on $\ch_N^i(\overline{\co}_M)$ (resp. on $\ch_N^j(\overline{\co}_M)$) we have $|\lambda'|=p^{i/2}, |\lambda''|=p^{j/2}.$ In particular, $\lambda'\neq\lambda'',$ unless $i=j.$

It follows that the identity (24) must split into several identities of the form
$$ Tr ((Fr^{*})^r,\ch_N^i(\overline{\co}_M))=a_i(\sqrt{p})^r), r\geq 1, i\geq 0,$$
where $a_i\in\bbq[\sqrt{p},\sqrt{p}^{-1}]$ is independent of $r.$ The Lemma is proved.
\end{proof}

Let $\mathbf{U}_n^{+}$ be the Hall algebra of nilpotent representations of the cyclic quiver. From Lemma 8.1.1, we may show that the basis of $\mathbf{U}_n^{+}$ defined by \cite{VV} is Lusztig's canonical bases of $\mathbf{U}_n^{+}.$
\begin{corollary} Let $\mathbf{U}_n^{+}$ as above, and let
$$\mathbf{B}_n =\{
\sum_{i,\overline{\co}_N\subsetneqq\overline{\co}_{M}}v^{i+\dim\co_N-\dim\co_{M}}\dim
\ch^i_{N}(IC({\overline{\co}_{M}},\overline{\bbq}_l))\lr{N}| M\in \bbe_\alpha,\alpha\in\bbn I\},$$
then $\mathbf{B}_n $ is a canonical bases of $\mathbf{U}_n^{+}.$
\end{corollary}

\subsection{Proof of Theorem~\ref{t:5.1.5}}

 Let $X=\overline{\co}_{P,M,{\bf
w},I}$ as in Theorem 6.1.1. By \cite{L4}, we known that $IC({\overline{\co}_P}),IC({\overline{\co}_{\cn_{{\bf w}}}}),\text{~and } IC({\overline{\co}_I})$ are strong pure. From Lemma 8.1.1, we get $IC({\overline{\co}_M})$ is strong pure.

Since $\co_{P,M,{\bf
w},I} = \co_P\star\co_M\star\cn_{{\bf w},3}\star\co_I.$ In the same way as Lemma 8.1.1, Theorem 6.1.1 are proved by induction on dimension vector and on orbits order.

\begin{corollary} Let $X_1,X_2,$ and $X_3$ as Theorem 5.1.1, then $\varphi_{X_1X_2}^{X_3}(x)\in\bbz[x].$
\end{corollary}
\begin{proof} Let $X=\overline{\co}_{P,M,{\bf
w},I},$ and $\langle X\rangle=\lr{P}*\lr{M}*E_{\textbf{w}\dz,3}*\lr{I}.$
By the corollary 6.2.1, we have
\begin{eqnarray}b_{\overline{\co}_{X_j}} =\langle X_j\rangle+\sum_{i,X\subsetneqq{\overline{\co}_{X_j}}}v^{i+\dim X-\dim\co_{X_j}}\dim
\ch^i_{X}(IC({\overline{\co}_{X_j}},\overline{\bbq}_l))\lr{X}.
\end{eqnarray}
Thus $\langle X_j\rangle=
b_{\overline{\co}_{X_j}}+\sum_{X\subsetneqq{\overline{\co}_{X_j}}}a_Xb_X, a_X\in \bbz[v,v^{-1}].$ Moreover, we have
\begin{eqnarray}\langle X_1\rangle*\langle X_2\rangle=(b_{\overline{\co}_{X_1}}+\sum_{X\subsetneqq{\overline{\co}_{X_1}}}a_Xb_X)*(b_{\overline{\co}_{X_2}}+\sum_{X\subsetneqq{\overline{\co}_{X_2}}}a_Xb_X)
\end{eqnarray}
 From  \cite{L4}, it follows that $\varphi_{X_1X_2}^{X_3}(x)\in\bbz[x].$
\end{proof}

\end{document}